\documentclass{amsart}
\usepackage{amssymb,amsthm}
\usepackage{hyperref}

\allowdisplaybreaks
\numberwithin{equation}{section}

\newcommand{\C}{{\mathbb{C}}}
\newcommand{\I}{{\mathfrak{I}}}
\newcommand{\N}{{\mathbb{N}}}
\newcommand{\R}{{\mathbb{R}}}

\newcommand{\Z}{{\mathbb{Z}}}
\newcommand{\op}{\textit{op}}

\newcommand{\qtq}[1]{\quad\text{#1}\quad}
\newcommand{\eps}{\varepsilon}

\newcommand{\ddt}{\frac{d\ }{dt}}
\newcommand{\ren}{\mathrm{ren}}

\DeclareMathOperator*{\wlim}{w-lim}
\DeclareMathOperator{\tr}{tr}

\DeclareMathOperator{\sech}{sech}
\let\Re=\undefined\DeclareMathOperator{\Re}{Re}
\let\Im=\undefined\DeclareMathOperator{\Im}{Im}
\let\xdet=\det
\let\det=\undefined\DeclareMathOperator{\det}{det}

\newtheorem{theorem}{Theorem}[section]
\newtheorem{prop}[theorem]{Proposition}
\newtheorem{lemma}[theorem]{Lemma}
\newtheorem{corollary}[theorem]{Corollary}
\theoremstyle{definition}
\newtheorem{definition}[theorem]{Definition}
\theoremstyle{remark}
\newtheorem*{remark}{Remark}

\begin{document}

\title{Orbital stability of KdV multisolitons in $H^{-1}$}

\author{Rowan Killip and Monica Vi\c{s}an}

\address
{Rowan Killip\\
Department of Mathematics\\
University of California, Los Angeles, CA 90095, USA}
\email{killip@math.ucla.edu}

\address
{Monica Vi\c{s}an\\
Department of Mathematics\\
University of California, Los Angeles, CA 90095, USA}
\email{visan@math.ucla.edu}

\begin{abstract}
We prove that multisoliton solutions of the Korteweg--de Vries equation are orbitally stable in $H^{-1}(\R)$.  We introduce a variational characterization of multisolitons that remains meaningful at such low regularity and show that all optimizing sequences converge to the manifold of multisolitons.  The proximity required at the initial time is uniform across the entire manifold of multisolitons; this had not been demonstrated previously, even in $H^1$.
\end{abstract}

\maketitle

%%%%%%%%%%%%%%%%%%%%%%%%%%%%%%%%%%%%%%%%%%%%%%%%%%%%%%%%%%%%%%%%%%%%%%%%%%%%%%%%%%%%%%%%%%%%%%%%%%%%%%%%%%%%%%%%%%%%%%%%%%%%%%
\section{Introduction}
%%%%%%%%%%%%%%%%%%%%%%%%%%%%%%%%%%%%%%%%%%%%%%%%%%%%%%%%%%%%%%%%%%%%%%%%%%%%%%%%%%%%%%%%%%%%%%%%%%%%%%%%%%%%%%%%%%%%%%%%%%%%%%

The history of the  Korteweg--de Vries equation
\begin{align}\label{KdV}\tag{KdV}
\ddt q = - q''' + 6qq'
\end{align}
is profoundly intertwined with the notion of solitary waves.  Indeed, the very goal of Korteweg and de~Vries \cite{KdV1895} was to explain the empirical observation of such waves.

The fact that \eqref{KdV} admits solutions of the form
\begin{align}\label{1 soliton}
q(t,x) = - 2\beta^2 \sech^2(\beta [x - 4\beta^2 t -x_0])
\end{align}
(for any $\beta>0$ and $x_0\in \R$) explains many aspects of solitary water waves, such as the relation between height and speed.  However, the very possibility of Scott Russell's famous chance encounter with such a wave tells us something more: It must be stable!

The question of stability was considered already by Boussinesq in \cite{Boo77}.  In addition to observing the conservation of both
\begin{equation}
P(q) := \int \tfrac12 q(x)^2\,dx \qtq{and} H(q) := \int \tfrac12 q'(x)^2 + q(x)^3 \,dx,
\end{equation}
he also notes that the solitary wave profile solves the Euler--Lagrange equation associated to the problem of optimizing $H$ subject to constrained $P$.

Now if the solitary wave were a non-degenerate minimum of $H$ at constrained $P$, then stability would follow immediately, following the Lyapunov model.  However, it is not!  The simple act of translation shows that it is at best a degenerate minimum.

In the pioneering paper \cite{MR0338584}, Benjamin proved the $H^1$-orbital stability of such solitary waves: Solutions close to a soliton profile at time zero remain close to a soliton profile at all times.  This variational approach is extremely robust and has seen countless applications since.  However, it does not directly give any information about the physical location of the soliton profile, nor how this evolves with time; this is the significance of the adjective `orbital'.

In numerical simulations of a discrete form of \eqref{KdV}, Kruskal and Zabusky \cite{KruskalZabusky} observed that solitary waves exhibit an even stronger form of stability: Pairs of solitary waves emerged from collisions with the same profile and speed with which they had entered.  Nevertheless, the two waves did interact; each was spatially shifted from its original trajectory.  This particle-like behavior led Kruskal and Zabusky to coin the name \emph{soliton}; they presciently appreciated that this was an exotic phenomenon.

We now understand that while the orbital stability of single solitary waves is rather common (and can often be proved variationally), stability under collisions is extremely peculiar.  The ultimate explanation for this behavior was the discovery that \eqref{KdV} is a completely integrable Hamiltonian system; see \cite{GGKM,MR0336122,MR0252826,MR0303132}.

Just as our notion of a solitary wave crystalizes around the concrete particular solutions \eqref{1 soliton} to \eqref{KdV}, so there is a family of special solutions to \eqref{KdV} that embody the behavior of collections of solitons:

\begin{definition}[Multisoliton solutions]\label{D:multisoliton}
Fix $N\geq 1$.  Given $N$ distinct positive parameters $\beta_1, \ldots, \beta_N$ and $N$ real parameters $c_1, \ldots, c_N$, let
\begin{align}\label{E:Qbc}
Q_{\vec \beta, \vec c}(x)= - 2\tfrac{d^2}{dx^2} \ln \det\bigl[A(x)\bigr]  
\end{align}
where $A(x)$ is the $N\times N$ matrix with entries
\begin{align}\label{matrix A}
A_{\mu\nu}(x) = \delta_{\mu\nu} + \tfrac{1}{\beta_\mu+\beta_\nu} e^{-\beta_\mu(x-c_\mu)-\beta_\nu(x-c_\nu)}.
\end{align}
The unique solution to \eqref{KdV} with initial data $q(0,x) = Q_{\vec \beta, \vec c}(x)$ is
\begin{align}\label{E:Qbc(t)}
q(t,x) = Q_{\vec \beta, \vec c(t)}(x) \qtq{where} c_n(t) = c_n + 4\beta_n^2 t   .
\end{align}
\end{definition}

The beautiful formula \eqref{E:Qbc} was originally derived in \cite{KayMoses} as a description of reflectionless potentials appearing in the one-dimensional Schr\"odinger equation.  With the discovery of the inverse-scattering approach, the significance of this result for \eqref{KdV} was noted by several authors; see \cite{GGKM,MR0336122,Hirota1971,MR0328386,WadatiToda,Zakharov1971}.   By analyzing these exact solutions, the authors confirmed the particle-like interactions, described the long-time asymptotics, and determined the (universal) spatial shifts.

The idea that these explicit solutions provide a justification for empirical observations is necessarily predicated (at the very least) on their stability.  Indeed, this question has attracted considerable attention over the years, as we shall discuss shortly.  Let us begin, however, with our own contribution to this question:

\begin{theorem}\label{T:main} 
Fix $N\geq 1$ and distinct positive parameters $\beta_1, \ldots, \beta_N$.  For every $\eps>0$ there exists $\delta>0$ so that for every initial data $q(0)\in H^{-1}(\R)$ satisfying
\begin{align*}
\inf_{\vec c\in \R^N} \|q(0)- Q_{\vec \beta, \vec c}\|_{H^{-1}}<\delta,
\end{align*}
the corresponding solution $q(t)$ to \eqref{KdV} satisfies
\begin{align*}
\sup_{t\in\R}\inf_{\vec c\in \R^N} \|q(t)- Q_{\vec \beta, \vec c}\|_{H^{-1}}<\eps.
\end{align*}
\end{theorem}

One virtue of this result is that it achieves the lowest regularity (in the $H^s$ scale) for which well-posedness is known \cite{KV} or possible \cite{MR2830706}. We shall also see that it is not difficult to recover higher-regularity results post factum:

\begin{corollary}\label{C:Hs}
Fix $s\in[-1,1]$, $N\geq 1$, and distinct positive parameters $\beta_1, \ldots, \beta_N$.  For every $\eps>0$ there exists $\delta>0$ so that
\begin{align}\label{Hs stab}
\inf_{\vec c\in \R^N} \|q(0) - Q_{\vec \beta, \vec c}\|_{H^s}<\delta
\ \implies\ %
\sup_{t\in\R}\inf_{\vec c\in \R^N} \|q(t)- Q_{\vec \beta, \vec c}\|_{H^s}<\eps.
\end{align}
\end{corollary}

The restriction $s\leq 1$ should not be taken too seriously.  Our goal is simply to illustrate two basic methods of raising the regularity without making the discussion too extensive, yet also recovering the important cases $L^2$ and $H^1$.

Let us now turn toward a discussion of prior work, after which we will discuss how the proof of Theorem~\ref{T:main} will proceed.  We do not intend to dwell on the question of well-posedness, since this is rather decoupled from the question of stability:  Proving an assertion like \eqref{Hs stab} only for Schwartz solutions still cuts to the heart of the matter; the Schwartz restriction can then be trivially removed once well-posedness in $H^s$ is known.  Indeed, Benjamin's work on $H^1$-stability should only grow in our estimation when we consider that well-posedness in $H^1$ was not achieved until many years later, \cite{MR1086966}.  Conversely, having obtained  well-posedness in $H^{-1}$ in \cite{KV}, it is timely to address the orbital stability in this space. 

It is also true that well-posedness alone provides little assistance in proving \eqref{Hs stab}.  Nevertheless, it has proved useful in the consideration of slightly weaker assertions, where $\delta$ is permitted to depend on the parameters $\vec c$ of the multisoliton nearest the initial data.  The manner in which it helps is this: Exact multisolitons resolve (as $t\to\pm\infty$) into essentially a linear combination of well-separated (and increasingly separated) simple solitary waves of the form \eqref{1 soliton}.  Thus, researchers may confine their analyses to this more favorable scenario and exploit well-posedness to cover the remaining compact time interval.

While Benjamin's argument \cite{MR0338584} was both extremely novel and compelling, it did contain some mathematical lacunae, particularly with regard to the treatment of the modulation parameters.  These issues were thoroughly addressed by Bona \cite{MR0386438}.  This approach was further developed to treat NLS and gKdV by Weinstein \cite{MR0820338}.

Orbital stability of the single soliton \eqref{1 soliton} in $L^2$ was only shown much more recently, by Merle and Vega \cite{MR1949297}.  These authors also show a form of asymptotic stability: one has $L^2$-convergence to a soliton profile in any bounded window traveling with the soliton.  Stronger forms of asymptotic stability such as global $L^2$ convergence are clearly forbidden by the conservative nature of the equation.  We should also note that it is not claimed that the solution is converging to a single solitary wave with fixed translation parameter $x_0$; indeed, subsequent analysis by Martel and  Merle \cite{MR2109467} shows that this cannot be guaranteed:   successive interactions with a large number of wide (and so $L^2$-small) solitary waves can lead to logarithmic divergence of the soliton trajectory from a straight line.

The Merle-Vega proof of $L^2$-orbital stability of single solitons combines the Miura map with orbital stability of the kink solutions proven in \cite{MR1831831}.  (While Zhidkov focusses on the NLS equation, his variational analysis employs only conservation laws common to mKdV.)

In the same paper \cite{MR0235310} that introduced the Lax pair, Lax also discusses two-soliton solutions with a view to explaining the properties of such waves observed in \cite{KruskalZabusky}.  His construction of such solutions is based on a differential equation derived from the polynomial conservation laws discovered earlier in \cite{MR0252826}.   While Lax does not explicitly express it thus in this paper (see \cite{MR0369963}, however), his equation arises as the Euler--Lagrange equation for optimizing the third conserved quantity with the first two constrained.  In general, $N$-solitons are critical points of the variational problem of optimizing the $(N+1)$ polynomial conserved quantity constrained by its $N$ predecessors (we exclude the Casimir $\int q\, dx$ from our enumeration). 

This constrained variational problem was analyzed by Maddocks and Sachs in \cite{MR1220540}.  They showed that multisolitons are in fact local minimizers.  The essential (and subtle) point addressed by these authors is to understand the Hessian of the highest-order conservation law on the manifold of multisolitons, both directly and restricted to directions parallel to the constraint manifold.

As the analysis in \cite{MR1220540} is localized in small neighbourhoods of the soliton profiles, it does not address either of the following questions:   
Are $N$-solitons global minimizers of this variational problem?  Are they the \emph{only} minimizers?  To the best of our knowledge, both questions remain open.  Theorem~\ref{T:char} below gives an affirmative answer to both questions for the variational description we employ.

Orbital stability of multisolitons in $H^1$ was shown by Martel, Merle, and Tsai in \cite{MR1946336}.  The principal part of the argument is showing that a system of well-separated solitons (ordered by speed) is future-stable.  Subsequently in \cite{MR2984057}, Alejo, Mu\~noz, and Vega proved orbital stability by using Gardner's generalization of the Miura map and applying the ideas of \cite{MR1946336} to the resulting Gardner equation.  These works do not yield orbital stability in the strong form of \eqref{Hs stab}; they rely on local well-posedness in the manner discussed earlier.  Additional information on the modulation parameters over this initial interval is obtained (a posteriori) in \cite{MR2339841}.

A different approach to orbital stability of solitons based on \emph{auto}B\"acklund transformations (which add or remove solitons) was demonstrated recently in \cite{MR2920823}.  This work proves a strong form of $L^2$-orbital stability of one-solitons for the focusing cubic NLS on the line by combining these transformations with stability of the zero solution.  This approach was substantially advanced in \cite{KochTataru}, where low-regularity orbital stability of NLS multisolitons (including the delicate case of multiple eigenvalues) was proved.  To the best of our knowledge, these ideas have not yet been applied to \eqref{KdV}.

Let us now turn to the topic of the methods to be employed in this paper.  Our discussion will be somewhat discursive since we shall take the time to introduce the central object of our methodology, the (doubly) renormalized perturbation determinant, as well as historical and contextual matters that we find instructive.
  
As we have discussed, the stability of multisolitons is historically (and physically) inseparable from the complete integrability of \eqref{KdV}.  The key question is how this complete integrability is to be exploited.

The long-standing approach, introduced already in \cite{GGKM}, is to employ the scattering theory of one-dimensional Schr\"odinger operators with the potential given by the \eqref{KdV} wave form at a fixed time.  Despite receiving a great deal of attention over the years (with much impetus taken from the study of KdV), there is currently no satisfactory theory of forward or inverse scattering in any $H^s$ space.  While non-trivial problems do attend low regularity, it is the slow decay associated with such spaces that is most devastating.  We are truly at a loss as to how to define the reflection coefficient or how to handle embedded eigenvalues and singular continuous spectrum.

The inverse-scattering technique is capable of providing extremely detailed long-time asymptotics for the class of solutions to which it is applicable; see \cite{MR2525595}, for example.  However, due to the difficulties outlined above, it has not yet yielded stability of even single-solitons in any $H^s$ space. 

While the reflection coefficient is fragile, it has long been appreciated that the transmission coefficient is much more robust.  One intuitive explanation for this is that the transmission coefficient actually represents the boundary values of a function meromorphic in the upper half-plane.  Analytically, it is preferable to consider the reciprocal $a(k;q)$ of the transmission coefficient.  This is holomorphic in the upper half-plane and its zeros precisely encode the discrete spectrum of the attendant Schr\"odinger operator.  The simplest description is as the Wronskian (divided by $2ik$) of the two Jost solutions.  
 
An alternate perspective on this function $a(k;q)$ was introduced by Jost and Pais \cite{MR0044404}.  They observed that it could be expressed as a Fredholm determinant.
In \cite[Chapter~5]{MR2154153}, Simon proves that 
\begin{equation}\label{JostPaisSimon}
a(k; q)= \det \bigl(1+ |q|^{\frac12} R_0(k) |q|^{-\frac12} q\bigr), \qtq{where} R_0(k) = (-\partial^2_x - k^2)^{-1},
\end{equation}
coincides with the Wronskian definition provided $\langle x\rangle^{1+\delta} q \in L^1$ with $\delta>0$.

Splitting $q$ across the two sides of $R_0$ is necessary if one wishes to treat $q$ with $L^1$-type singularities: neither $R_0 q$ nor $q R_0$ could be guaranteed to be bounded under Simon's hypothesis.  However, it turns out to be wiser to factor the free resolvent $R_0$, placing a square-root of this operator on either side of $q$; as we shall see, this will permit potentials with much more severe singularities.  On the other hand, one still needs strong decay hypotheses on $q$; for otherwise, the determinant would not be defined.

The second layer of renormalization needed to treat $q\in H^{-1}$ employs the renormalized determinant introduced by Hilbert \cite{Hilbert}; see \cite[Chapter~9]{MR2154153}.  Combining these two ideas, we are led to consider the following: For $k\in \C^+ = \{z \in \C : \Im z >0\}$ and Schwartz-class $q$, 
\begin{align}\label{a ren}
a_{\ren} (k;q) := \det_2 \bigl(1+ \sqrt{R_0(k)}\, q\, \sqrt{R_0(k)} \bigr)= a(k;q)\exp\Bigl\{-\tfrac i{2k}\int q(x)\,dx\Bigr\}.
\end{align}
The square-root of the resolvent is defined via analytic continuation from the case $k=i\kappa$ with $\kappa>0$, in which case $R_0$ is positive definite (and we take the positive definite square-root).

To the best of our knowledge, this quantity was first considered by Rybkin.  In \cite{MR2683250}, he used it to give the first proof of a priori $H^{-1}$ bounds for solutions to \eqref{KdV}.   This approach was developed independently in \cite{KVZ}; alternate approaches to such a priori bounds can be found in \cite{MR3400442,MR3874652}.

The fact that $a_{\ren}$ extends continuously (indeed real-analytically) from $q\in\mathcal S$ to merely $q\in H^{-1}$ rests on the basic theory of such regularized determinants and the Hilbert--Schmidt estimate
\begin{align}\label{R I2}
\Bigl\| \sqrt{R_0(k)}\, q\, \sqrt{R_0(k)} \Bigr\|^2_{\mathfrak I_2} &\leq \frac{|k|}{[\Im k]^2}\int \frac{|\hat q(\xi)|^2}{\xi^2+4|k|^2}\,d\xi.
\end{align}
Indeed, the mapping $A\mapsto \det_2(1+A)$ is a complex-analytic function on $\I_2$ and
\begin{align}\label{HS det2 bnd}
\bigl| 1 - \det_2(1+A) \bigr| \lesssim  \| A\|_{\I_2} \exp\bigl\{  \| A\|_{\I_2}^2 \bigr\} ;
\end{align}
see \cite{MR2154153} for details.  Our justification for the bound \eqref{R I2} is quite simple.  We use the ideal property and the elementary bound
$$
|\xi^2-k^2|^{-1} \leq \tfrac{|k|}{\Im k} (\xi^2+|k|^2)^{-1} \qtq{for all} \xi\in\R
$$
to reduce matters to the $\kappa=|k|$ case of
\begin{align}\label{R I2 kappa}
\Bigl\| \sqrt{R_0(i\kappa)}\, q\, \sqrt{R_0(i\kappa)} \Bigr\|^2_{\mathfrak I_2} & = \frac{1}{\kappa} \int \frac{|\hat q(\xi)|^2}{\xi^2+4\kappa^2}\,d\xi
	\qtq{for all} \kappa>0.
\end{align}

In view of the importance of \eqref{R I2 kappa} for what follows, it will also be convenient to employ the notation
$$
\| f \|_{H^{-1}_\kappa}^2 := \int \frac{|\hat f(\xi)|^2}{\xi^2+4\kappa^2}\,d\xi.
$$

With these preliminaries set, we may now give our variational characterization of multisolitons:

\begin{theorem}[Variational characterization of multisolitons]\label{T:char}
Fix $N\geq 1$ and distinct positive parameters $\beta_1, \ldots , \beta_N$.  If $q\in H^{-1}$ satisfies
\begin{align}\label{aren 0}
a_{\ren}(k;q)=0 \qtq{for all} k\in\{i\beta_m:\ 1\leq m\leq N\},
\end{align}
then
\begin{align}\label{a_ren is max}
a_\ren(i\kappa,q) \leq \exp\biggl\{\, \sum_{m=1}^N \ln\bigl(\tfrac{\kappa-\beta_m}{\kappa+\beta_m}\bigr) +\tfrac{2\beta_m}{\kappa}\biggr\}  \quad\text{for all}\quad \kappa\geq 1 + \|q\|_{H^{-1}}^2 .
\end{align}
If equality holds in \eqref{a_ren is max} for any one such $\kappa$, then $q=Q_{\vec \beta, \vec c}$ for some $\vec c\in \R^N$.
\end{theorem}

By itself, Theorem~\ref{T:char} does not provide stability: one would also need to know that profiles that almost optimize \eqref{a_ren is max} are close to actual optimizers (i.e., to multisolitons).  This leaves us with a very clear ambition of a purely variational character: prove that optimizing sequences converge to the manifold of multisolitons.

We cannot expect optimizing sequences to have convergent subsequences --- the manifold of optimizers is not compact!  This problem arises already in the case of single solitons, due to the translation symmetry.  In the one-soliton case, compactness can be restored by incorporating translations.  This approach was convincingly demonstrated by Cazenave and Lions \cite{MR0677997}, who proved orbital stability of ground-state solitary waves for a variety of NLS-like equations.  Their paper is a major inspiration for what follows.

In the multisoliton case, compactness cannot be restored by translation alone.  Indeed the long-time dynamics of the multisolitons themselves is to break into asymptotically well-separated one-solitons.  We need a profile decomposition!  However, unlike most applications of this concentration-compactness technique, there is no sub-additivity in our problem: dichotomy must be embraced, not refuted.  As we will discuss, this is just one of several subtle aspects to our implementation of this classic concentration-compactness device.

We should note that the scenario of asymptotically well-separated one-solitons is not the only manner in which dichotomy can arise for optimizing sequences (or indeed sequences of optimizers).  One may have asymptotically well-separated \emph{multi}solitons.  This `gas of molecules' scenario will be analyzed in Section~\ref{S:MS}, where we show that a linear combination of well-separated multisolitons can be well-approximated by a single exact multisoliton.

Further ways in which our concentration-compactness analysis diverges from the other examples we know are (i) we are working in trace ideals, not Lebesgue spaces;
(ii) while we do have local compactness, this is non-quantitative arising from mere equicontinuity; and (iii) the constraints are apportioned across the profiles in an exotic manner.  We will discuss each of these in succession.

Trace ideals (which are also known as non-commutative $\ell^p$ spaces) have an additional defect of compactness beyond those of sequence spaces, namely, unitary conjugation.  This in an infinite-dimensional group.

Local compactness is not a prerequisite for concentration-compactness methods; indeed, with the incorporation of scaling parameters, such methods have proven to be extremely useful in scaling-critical problems.  Nevertheless, in the examples we know, local compactness is obtained from the Rellich--Kondrashov Theorem.  In our case, however, there is no such quantitative principle.  We will be able to show that individual optimizing sequences are equicontinuous, but nothing more.

In the standard analyses, a constraint, such as on the total $L^2$ norm, is apportioned across the profiles in an additive manner: the mass of the sequence is the sum of the masses of the profiles, plus that of the remainder.  In our case, the constraints are vanishing of the perturbation determinant.  In Section~\ref{S:OS}, we will see that the profiles attendant to optimizing sequences share the constraints in a different way: different profiles satisfy different \emph{subsets} of the constraints.

The paper is organized as follows:  In Section~\ref{S:2}, we first develop the theory of the perturbation determinant a little further.  We then use this to prove Theorem~\ref{T:char}.  Our approach is this: Building on the existing theory of Schwartz-class potentials, we show that the upper-bound \eqref{a_ren is max} holds across all $q\in H^{-1}$.  Having first proved linear independence of the gradients of the constraints, we may analyze the  case of equality using the Euler--Lagrange equation.  Using this device, we show that optimizers are, in fact, Schwartz class.  We may then appeal to classical inverse scattering to deduce that $q$ is an exact multisoliton.

In Section~\ref{S:MS}, we show that well-separated linear combinations of multisolitons (which may arise as optimizing sequences) can be approximated by a single multisoliton.  This is notationally very cumbersome; nevertheless, we hope that the virtues of deforming $x$ into the complex plane and exploiting the determinantal relation \eqref{inductive Cauchy} shine through.

In Section~\ref{S:CC}, we develop a profile decomposition attendant to the functional $q\mapsto \alpha(\kappa;q)$, defined in \eqref{alpha defn}, applied to bounded and equicontinuous sequences in $H^{-1}$.  Structurally speaking, our approach is the one we advanced in \cite{MR3098643}, namely, to first prove an inverse inequality and then employ this inductively to extract profiles.

In Section~\ref{S:OS}, we prove Theorem~\ref{T:main}, arguing by contradiction. If the theorem were to fail, then there would exist a sequence of solutions $q_n$ so that the initial data $q_n(0)$ converges to the manifold of solitons, and a sequence of times $t_n$ so that $q_n(t_n)$ does not converge to the manifold of solitons.  Using the fact that $\alpha(q;\kappa)$ is conserved under the flow, we show that $q_n$ is an optimizing sequence for the variational problem described in Theorem~\ref{T:char}.  (Actually, this is not quite correct, the zeros may be slightly displaced.)  We then employ the profile decomposition of Section~\ref{S:CC} to show (after some work) that the optimizing sequence can be approximated by a linear combination of well-separated multisolitons.  This suffices to reach a contradiction because of the analysis in Section~\ref{S:MS}.

We prove Corollary~\ref{C:Hs} in Section~\ref{S:Hs}.   In doing so, we illustrate two basic methods for raising the regularity: (i) employing polynomial conservation laws and (ii) exploiting equicontinuity of orbits.  Both methods are applicable beyond the range claimed in Corollary~\ref{C:Hs}; however, the details become increasingly cumbersome as the regularity $s$ grows.

\subsection*{Acknowledgements} R. K. was supported by NSF grant DMS-1856755 and M.~V. by grant DMS-1763074.

%%%%%%%%%%%%%%%%%%%%%%%%%%%%%%%%%%%%%%%%%%%%%%%%%%%%%%%%%%%%%%%%%%%%%%%%%%%%%%%%%%%%%%%%%%%%%%%%%%%%%%%%%%%%%%%%%%%%%%%%%%%%%%
\section{Variational characterization of multisolitons}\label{S:2}
%%%%%%%%%%%%%%%%%%%%%%%%%%%%%%%%%%%%%%%%%%%%%%%%%%%%%%%%%%%%%%%%%%%%%%%%%%%%%%%%%%%%%%%%%%%%%%%%%%%%%%%%%%%%%%%%%%%%%%%%%%%%%%

The ultimate goal of this section is to prove Theorem~\ref{T:char}.  This will proceed in several stages.  First, we discuss the logarithm of $a_\ren$. Then we show that \eqref{a_ren is max} holds, first for Schwartz-class $q$ and then for general $q\in H^{-1}$.  The climax of the proof is showing that all $H^{-1}$ optimizers are, in fact, Schwartz class and then using this information to show that they must be multisolitons.

\begin{lemma}\label{L:alpha defn}
For $q\in H^{-1}$ and $\kappa\geq 1 + \|q\|_{H^{-1}}^2$, the series
\begin{align}\label{alpha defn}
\alpha(q;\kappa) := \sum_{\ell=2}^\infty \tfrac{1}{\ell} (-1)^\ell \tr\Bigl\{ \Bigl( \sqrt{R_0(i\kappa )}\, q\, \sqrt{R_0(i\kappa )}\, \Bigr)^\ell \Bigr\}
\end{align}
converges and
\begin{align}\label{arenalpha}
a_\ren(q;i\kappa) = \exp\{ - \alpha(q;\kappa)\}.
\end{align}
Moreover,
\begin{align}\label{alpha to L2}
\liminf_{\kappa\to\infty} 8\kappa^3 \alpha(i\kappa;q) = \|q\|_{L^2}^2,
\end{align}
with the understanding that  LHS\eqref{alpha to L2} is infinite if $q\notin L^2$.
\end{lemma}

\begin{proof}
Convergence of the series \eqref{alpha defn} under this hypothesis on $\kappa$ follows immediately from \eqref{R I2 kappa}.  That exponentiating this series yields the renormalized determinant is well known; indeed, this is little more than the Newton--Girard relation between elementary and power-sum symmetric functions.

Employing \eqref{R I2 kappa} in the series \eqref{alpha defn} shows
\begin{align}\label{37deg}
\biggl| 8\kappa^3\alpha(q;\kappa) - \int \frac{4\kappa^2 |\hat q(\xi)|^2}{\xi^2+4\kappa^2}\,d\xi \biggr| \leq \frac{\| q\|_{H^{-1}}}{\sqrt{\kappa} -  \| q\|_{H^{-1}}} \int \frac{4\kappa^2 |\hat q(\xi)|^2}{\xi^2+4\kappa^2}\,d\xi,
\end{align}
from which \eqref{alpha to L2} follows immediately.
\end{proof}

Incidentally, we note the inequality \eqref{37deg} is actually the basis of the proof of a priori $H^{-1}$ bounds.  Indeed, combining this with a simple bootstrap argument shows that for $\kappa\geq 1 + 64\|q(0)\|_{H^{-1}_\kappa}^2$ and any $t\in\R$,
\begin{align}\label{37deg'}
\frac23 \int \frac{|\hat q(t, \xi)|^2}{\xi^2+4\kappa^2}\,d\xi \leq 2\kappa\alpha(\kappa; q(t)) = 2\kappa\alpha(\kappa; q(0))
	\leq \frac87 \int \frac{|\hat q(0, \xi)|^2}{\xi^2+4\kappa^2}\,d\xi .
\end{align}

Let us now recall some known facts about the reciprocal transmission coefficient $a(k;q)$ in the case $q$ is of Schwartz class.  The basic analytical facts listed below can be easily derived from the Wronskian definition of $a(k;q)$ and rigorous proofs can be found in many basic texts on scattering theory.  The claim \eqref{a soliton} is more serious.  While many introductory texts on the theory of solitons give at least a formal derivation of \eqref{E:Qbc} from the assumption that $a(k;q)$ takes the stated form, a rigorous treatment requires considerable care, especially on the question of uniqueness.  We recommend the paper \cite{MR0512420} of Deift and Trubowitz for a complete and self-contained presentation of the following (under rather weaker hypotheses):

\begin{prop}\label{P:DT}
Fix $q\in \mathcal S$.  Then $a(k;q)$ extends continuously to the closed upper half-plane.  It has finitely many zeroes in $\C^+$, all of which are simple and located on the imaginary axis.  Moreover,
\begin{gather}
|a(k;q)| \geq 1 \quad\text{for all} \quad k\in \R, \\
|a(k;q)- 1| = O\bigl(\tfrac1{|k|}\bigr) \quad\text{as} \quad |k|\to \infty \quad\text{uniformly for}\quad \Im k\geq 0,
\end{gather}
and we have the symmetry
\begin{equation}\label{conj symm}
\overline{ a(k;q) }= a(-\bar k; q) \quad\text{for all} \quad k\in \C^{+}.
\end{equation}

Finally, given distinct $\beta_1,\ldots,\beta_N\in(0,\infty)$ and $q\in \mathcal S$, 
\begin{align}\label{a soliton}
a(k; q)= \prod_{m=1}^N \frac{k-i\beta_m}{k+i\beta_m} \iff q\in\bigl\{ Q_{\vec \beta, \vec c}:\, \vec c\in \R^N\bigr\}.
\end{align}
\end{prop}

This does not address the value of the renormalized perturbation determinant for such multisolitons.  The missing ingredient is the following:
\begin{equation}\label{trace form0}
\int Q_{\vec \beta, \vec c}(x)\, dx= -\sum_{m=1}^N 4\beta_m.
\end{equation}
This is proved in both \cite{MR0336122} and \cite{MR0303132}.  One simple approach that explains the additive structure of RHS\eqref{trace form0} is this:  As the $\int q$ is conserved by the flow, the value of LHS\eqref{trace form0} can be determined from the value for $N$ well-separated single solitons.  Alternately, one may deduce this by comparing the large-$k$ asymptotics of LHS\eqref{a soliton} with those of $a(\kappa;q)$.  From the same references or by the same method, one can also find
\begin{align}\label{conserv of multi}
P\bigl(Q_{\vec\beta,\vec c}\bigr) = \tfrac{8}{3}\sum_m \beta_m^3 \qtq{and} H\bigl(Q_{\vec\beta,\vec c}\bigr) =-\tfrac{32}5\sum_m\beta_m^5.
\end{align}

Combining \eqref{a ren}, \eqref{arenalpha}, and \eqref{trace form0} shows
\begin{gather}
a_{\ren}(k; Q_{\vec \beta, \vec c}) = \prod_{m=1}^N \frac{k-i\beta_m}{k+i\beta_m} e^{{2i\beta_m}/{k}} \qtq{for all} k\in \C^+, \label{a ren Qbc} \\
\alpha(\kappa;Q_{\vec \beta, \vec c}) = -\sum_{m=1}^N \tfrac{2\beta_m}{\kappa} + \ln\bigl(\tfrac{\kappa-\beta_m}{\kappa+\beta_m}\bigr)
	\qtq{for all} \kappa\geq 1 + \|q\|_{H^{-1}}^2 .  \label{alpha Qbc}
\end{gather}
As Lemma~\ref{L:alpha defn} guarantees that $a_\ren$ is non-vanishing for $\kappa\geq 1 + \|q\|_{H^{-1}}^2$, the restriction on $\kappa$ guarantees $\kappa > \sup_m \beta_m$ and consequently, that RHS\eqref{alpha Qbc} is positive:
\begin{align}\label{G}
G\bigl(\tfrac{\beta_m}\kappa\bigr) :=  - \Bigl[ \tfrac{2\beta_m}{\kappa} + \ln\bigl(\tfrac{\kappa-\beta_m}{\kappa+\beta_m}\bigr) \Bigr]
	= \sum_{\ell\geq1}  \tfrac{2}{2\ell+1} \bigl(\tfrac{\beta_m}\kappa\bigr)^{2\ell+1} \geq 0.
\end{align}

Recalling \eqref{arenalpha}, we see that multisolitons achieve equality in \eqref{a_ren is max}.  We next show that this is indeed a bound for all $q$.  This will be done in two steps: first for $q\in\mathcal S$ and then for $q\in H^{-1}$:

\begin{prop}\label{P:min alpha}
For $q\in \mathcal S$ and $\kappa\geq 1 + \|q\|_{H^{-1}}^2$,
\begin{align}\label{E:min alpha}
\alpha(\kappa;q) \geq -\sum_{m=1}^N \Bigl[ \ln\bigl(\tfrac{\kappa-\beta_m}{\kappa+\beta_m}\bigr) +\tfrac{2\beta_m}{\kappa} \Bigr],
\end{align}
where $\{i\beta_m:\ 1\leq m\leq N\}$ enumerates the $0\leq N<\infty$ zeros of $a_\ren$.
\end{prop}

\begin{proof}
Proposition~\ref{P:DT} shows that $a(k;q)$ has only finitely many zeros and all are simple.  Using these, we build the Blaschke product
$$
B(k)=\prod_{m=1}^N \frac{k-i\beta_m}{k+i\beta_m}.
$$
In the case $a(k;q)$ has no zeros, $B(k)\equiv 1$.

Using Proposition~\ref{P:DT} again, we see that $k\mapsto \ln\bigl| \frac{a(k;q)}{B(k)} \bigr|$ is harmonic on $\C^+$ and extends continuously to $\partial \C^+$.  Moreover,  
\begin{align}\label{prop of frac}
\ln\bigl| \tfrac{a(k;q)}{B(k)} \bigr|\geq 0 \quad \text{for all  $k\in \R$} \qquad \text{and}\qquad\ln\bigl| \tfrac{a(k;q)}{B(k)} \bigr|= O\bigl(\tfrac1{|k|}\bigr) \quad \text{as  $|k|\to \infty$}.
\end{align}
It follows from the maximum principle that this function is non-negative throughout~$\C^+$. 

The Herglotz Representation Theorem (cf. \cite[Theorem~3, \S59]{AG}) then guarantees 
\begin{align}\label{rep}
\ln\bigl[\tfrac{a(k;q)}{B(k)}\bigr] = -i \int_\R \tfrac{d\mu(t)}{t-k},
\end{align}
for some finite positive measure $d\mu$ on $\R$.  This measure is also even under $t\mapsto -t$; this is inherited from the symmetry \eqref{conj symm} enjoyed by both $a(k;q)$ and $B(k)$.

In this way, we see that for $\kappa\geq 1 + \|q\|_{H^{-1}}^2$,
\begin{align*}
-\ln a(i\kappa ;q) &= -\sum_m \ln \bigl(\tfrac{\kappa-\beta_m}{\kappa+\beta_m}\bigr)  + i\int\tfrac{t+i\kappa}{t^2+\kappa^2}\, d\mu(t)= -\sum_m \ln \bigl(\tfrac{\kappa-\beta_m}{\kappa+\beta_m}\bigr)  -\kappa\int\tfrac{d\mu(t)}{t^2+\kappa^2}.
\end{align*}
On the other hand, \eqref{a ren}, \eqref{arenalpha}, and \eqref{alpha to L2} show that as $\kappa\to\infty$,
\begin{align*}
\Bigl|-\ln a(i\kappa ;q) + \tfrac1{2\kappa}\int q(x)\, dx\Bigr| &= O(\kappa^{-3}).
\end{align*}
Combining these two observations, we deduce that
\begin{equation}\label{9:59}
\begin{aligned}
\int q(x)\, dx &= \lim_{\kappa\to \infty} \biggl[ \sum_m 2\kappa \ln \bigl(\tfrac{\kappa-\beta_m}{\kappa+\beta_m}\bigr)
	+ \int \tfrac{2\kappa^2}{t^2+\kappa^2}\, d\mu(t) \biggr] \\
&= -4\sum_m\beta_m + 2\int d\mu(t)
\end{aligned}
\end{equation}
and thence that\begin{align}\label{alpha with mu}
\alpha(\kappa;q) = -\ln a_{\ren}(i\kappa ;q) &= -\sum_{m=1}^N \Bigl[\ln\bigl(\tfrac{\kappa-\beta_m}{\kappa+\beta_m}\bigr) +\tfrac{2\beta_m}{\kappa}\Bigr] +\int\tfrac{t^2}{\kappa(t^2+\kappa^2)}\,d\mu(t),
\end{align}
for all $\kappa\geq 1 + \|q\|_{H^{-1}}^2$.  The claim \eqref{E:min alpha} follows since $d\mu\geq 0$.
\end{proof}

\begin{corollary}\label{C:min alpha}
Fix $N\geq 0$ and distinct positive parameters $\beta_1, \ldots , \beta_N$. Assume that $q\in H^{-1}$ satisfies
$$
a_{\ren}(i\beta_m;q)=0 \quad\text{for all}\quad 1\leq m\leq N.
$$
Then for $\kappa\geq 1 + \|q\|_{H^{-1}}^2$ we have
\begin{align}\label{E:C:min alpha}
\alpha(\kappa;q) \geq -\sum_{m=1}^N \ln\bigl(\tfrac{\kappa-\beta_m}{\kappa+\beta_m}\bigr) +\tfrac{2\beta_m}{\kappa}.
\end{align}
Moreover, if equality holds in \eqref{E:C:min alpha} for one such $\kappa$ then it holds for all such $\kappa$.
\end{corollary}

\begin{proof}
Let $\{f_n\}_{n\geq 1}$ be a sequence of Schwartz functions that converge to $q$ in $H^{-1}$.  As the renormalized perturbation determinant is continuous on $H^{-1}$, we have
$$
\lim_{n\to \infty} a_{\ren}(k;f_n)=a_{\ren}(k;q) \quad\text{uniformly on compact subsets of $\C^+$}.
$$
Using Hurwitz's theorem and  \eqref{a ren}, we deduce that for each $1\leq m\leq N$ and $n$ sufficiently large there exits distinct $\beta_m^{(n)}$ so that
\begin{equation}\label{2:03}
a(i\beta_m^{(n)};f_n) = 0 \quad \text{and} \quad \lim_{n\to \infty}\beta_m^{(n)}=\beta_m.
\end{equation}

In view of \eqref{E:min alpha} and the positivity \eqref{G}, we find that
\begin{align}\label{2:04}
\alpha(\kappa;f_n) \geq -\sum_{m=1}^N \ln\Bigl(\tfrac{\kappa-\beta_m^{(n)}}{\kappa+\beta_m^{(n)}}\Bigr) +\tfrac{2\beta_m^{(n)}}{\kappa}\xrightarrow[n\to \infty]{} -\sum_{m=1}^N \ln\Bigl(\tfrac{\kappa-\beta_m}{\kappa+\beta_m}\Bigr) +\tfrac{2\beta_m}{\kappa}.
\end{align}
The claim \eqref{E:C:min alpha} now follows because $\alpha$ is continuous on $H^{-1}$.

Suppose now that equality holds in \eqref{E:C:min alpha} for some single value $\kappa_0\geq 1 + \|q\|_{H^{-1}}^2$. Let us write $d\mu_n$ for the measure representing $a(k;f_n)$ in the sense of \eqref{rep}.  It then follows from \eqref{alpha with mu} and \eqref{2:03} that
$$
\int\tfrac{t^2}{\kappa_0(t^2+\kappa_0^2)}\,d\mu_n(t) \to 0  \qtq{and thence that} \int\tfrac{t^2}{\kappa(t^2+\kappa^2)}\,d\mu_n(t) \to 0
$$
for every $\kappa >0$.  This in turn guarantees that equality holds in \eqref{E:C:min alpha} for every $\kappa\geq 1 + \|q\|_{H^{-1}}^2$.
\end{proof}

We are now ready to realize the ultimate goal of this section:

\begin{proof}[Proof of Theorem~\ref{T:char}]
In view of Corollary~\ref{C:min alpha}, it remains to show that if \eqref{aren 0} holds and
\begin{align}\label{alpha sat}
\alpha(\kappa;q) = -\sum_{m=1}^N \ln\bigl(\tfrac{\kappa-\beta_m}{\kappa+\beta_m}\bigr) +\tfrac{2\beta_m}{\kappa} \qtq{for all} \kappa\geq 1 + \|q\|_{H^{-1}}^2,
\end{align}
then $q=Q_{\vec \beta,\vec c}$ for some choice of $\vec c \in \R^N$.
The first step in the proof will be to show that all such optimizers $q$ belong to Schwartz class.  In the second step, we will prove that
\begin{align}\label{a phi goal} 
a(k; q)= \prod_{m=1}^N \frac{k-i\beta_m}{k+i\beta_m}.
\end{align}
In view of Proposition~\ref{P:DT}, this implies that $q$ is a multisoliton, thus completing the proof of the theorem.

From \eqref{alpha to L2}, \eqref{G}, and \eqref{alpha sat}, we see already that $q\in L^2$.   We will get further regularity and decay by studying the Euler--Lagrange equation satisfied by $q$.  To begin, we note that since $a_{\ren}(i\beta_m;q)=0$, there exist $\phi_m\in L^2$ such that
$$
\bigl(1+ \sqrt{R_0(i\beta_m)}q\sqrt{R_0(i\beta_m)}\bigr)\phi_m=0 \qtq{and} \|\phi_m\|_2=1.
$$
Writing $\psi_m:=\sqrt{R_0(i\beta_m)}\phi_m$ we obtain
\begin{align}\label{psi eq}
\bigl(-\partial^2 + q+\beta_m^2\bigr)\psi_m=0 \qtq{and} \|\psi_m\|_{H^1_{\beta_m}}=1.
\end{align}
Note also that the eigenvalue $-\beta_m^2$ must be simple.  Indeed, if there were two linearly independent eigenvectors, this would yield linearly independent solutions to \eqref{psi eq}, both belonging to $H^1$; this is inconsistent with constancy of the Wronskian.

As $q\in L^2$, we see that \eqref{psi eq} implies that $\psi_m\in H^2$ and so $\psi_m^2\in L^1\cap H^2$.  Moreover, a quick computation shows that
\begin{align}\label{op to psim2}
\bigl(-\partial^3+2\partial q + 2q\partial+4\kappa^2\partial\bigr) \psi_m^2 = 4(\kappa^2-\beta_m^2)(\psi_m^2)'
\end{align}
in $H^{-1}$ sense.

Next, we claim that the functions $\{\psi_m^2\}_{m=1}^N$ are linearly independent.  Indeed, assume (towards a contradiction) that there were a \emph{minimal} collection
$\Lambda\subseteq \{1, \ldots, N\}$ such that
\begin{align}\label{ld}
\sum_{m\in \Lambda} c_m \psi_m^2=0 \qquad\text{with}\qquad c_m\neq 0 \qtq{for all} m\in \Lambda.
\end{align}
Fixing some $n\in \Lambda$ and applying $\bigl(-\partial^3+2\partial q + 2q\partial+4\beta_n^2\partial\bigr)$ to \eqref{ld} and using \eqref{op to psim2} we obtain
$$
\sum_{m\neq n\in \Lambda}c_n (\beta_m^2-\beta_n^2)(\psi_m^2)'=0.
$$
As $\beta_1, \ldots, \beta_N$ are distinct and $\psi_m^2$ decay at infinity, this contradicts the minimality of the collection $\Lambda$.

The functions $\psi_m^2$ represent the gradients of the constraints $a_{\ren}(i\beta_m; \phi)=0$.  Indeed,
$$
\tfrac{\delta}{\delta q} a_{\ren}(i\beta_m;q) = \xdet_{\substack{\phi_m^{\perp}}}{\!}_2 \bigl(1+ \sqrt{R_0(i\beta_m)}q\sqrt{R_0(i\beta_m)}\bigr) \, \psi_m^2,
$$
where the subscript on $\det_2$ indicates the Hilbert space over which the renormalized determinant is computed.  Concretely, in this case this is the Hilbert space of functions orthogonal to $\phi_m$.  As the eigenvalues $-\beta_m^2$ are simple, the renormalized determinant over $\phi_m^\perp$ is non-zero.

The gradient of $\alpha$ is easily derived from the series \eqref{alpha defn}:
$$
\tfrac{\delta}{\delta q}\alpha(\kappa; q)=\tfrac1{2\kappa} - g(\kappa;q),
$$
where $g(\kappa;q)$ is the diagonal Green's function.  This is discussed in greater detail in \cite{KV}.  As $q\in L^2$, \cite[Proposition~A.2]{KV} shows that $\tfrac1{2\kappa} - g(\kappa;\phi)\in H^2$; we also have the long-known identity
$$
\bigl(-\partial^3+2\partial q + 2q \partial+4\kappa^2\partial\bigr) g(\kappa;q)=0 
$$
which holds in $H^{-1}$ sense (cf. \cite[Proposition~2.3]{KV}).

As the gradients $\psi_m^2$ of the constraints have been shown to be linearly independent, we deduce that the optimizer $q$ satisfies the Euler--Lagrange equation
\begin{align}\label{EL}
\tfrac1{2\kappa} - g(\kappa;q) = \sum_{m=1}^N\lambda_m\psi_m^2
\end{align}
for each $\kappa\geq 1+\|q\|_{H^{-1}}^2$ and some ($\kappa$-dependent) multipliers $\lambda_1, \ldots, \lambda_N\in \R$.  
Consequently, applying $\bigl(-\partial^3+2\partial q + 2q \partial+4\kappa^2\partial\bigr)$ to \eqref{EL} and using \eqref{op to psim2}, we deduce that
\begin{align*}
\tfrac1\kappa q' = \sum_{m=1}^N 4\lambda_m (\kappa^2-\beta_m^2)(\psi_m^2)' .
\end{align*}
However, $q\in L^2$ and $\psi_m^2\in H^2$; thus
\begin{align}\label{q from psi}
\tfrac1\kappa q = \sum_{m=1}^N 4\lambda_m (\kappa^2-\beta_m^2)\psi_m^2
\end{align}
and so $q\in H^2$.  By alternately applying \eqref{psi eq} and \eqref{q from psi}, we deduce that $q$ is infinitely smooth.

From \eqref{q from psi} we see that $q\in L^1$.  It then follows from \eqref{psi eq} that each eigenfunction decays exponentially; see \cite[\S3.8]{MR0069338}.  Applying \eqref{q from psi} again we deduce that $q$ decays exponentially.  Thus $q\in \mathcal S$.

It remains to prove \eqref{a phi goal}.  Now that we know $q\in\mathcal S$, we may deploy the technology used in the proof of Proposition~\ref{P:min alpha}.  First we note that \eqref{E:min alpha} and the positivity \eqref{G} guarantee that $a_\ren$ has no zeros beyond those prescribed in \eqref{aren 0}.  In this way, the representation \eqref{rep} yields
$$
a_\ren(k; q) = \exp\biggl\{ - i \int \frac{d\mu(t)}{t-k} \biggr\} \cdot \prod_{m=1}^N \frac{k-i\beta_m}{k+i\beta_m}
$$
for some finite positive measure $d\mu$ on $\R$.  On comparing \eqref{alpha with mu} and \eqref{alpha sat}, we see that any mass $d\mu$ has must be concentrated at the origin.  Combining this observation with \eqref{a ren}, we deduce that
\begin{equation*}
a(k;q) = \exp\biggl\{ \frac{i}{k} \int d\mu + \frac{i}{2k} \int q\,dx \biggr\} \prod_{m=1}^N \frac{k-i\beta_m}{k+i\beta_m} .
\end{equation*}
As the holomorphic function $a(k;q)$ admits a continuous extension to $\partial \C^+$, this forces $\int q(x)\, dx = - 2 \int d\mu$ and so \eqref{a phi goal} holds.
\end{proof}

%%%%%%%%%%%%%%%%%%%%%%%%%%%%%%%%%%%%%%%%%%%%%%%%%%%%%%%%%%%%%%%%%%%%%%%%%%%%%%%%%%%%%%%%%%%%%%%%%%%%%%%%%%%%%%%%%%%%%%%%%%%%%%
\section{Molecular decomposition of multisolitons}\label{S:MS}
%%%%%%%%%%%%%%%%%%%%%%%%%%%%%%%%%%%%%%%%%%%%%%%%%%%%%%%%%%%%%%%%%%%%%%%%%%%%%%%%%%%%%%%%%%%%%%%%%%%%%%%%%%%%%%%%%%%%%%%%%%%%%%

The principal goal of this section is to prove that linear combinations of well-separated multisolitons are close to the manifold of multisolitons.  We refer to this as a molecular decomposition building on the analogy of one-solitons to atoms and of multisolitons to molecules.  In fact, we will see that the eigenvalue parameters $\vec \beta^j$ of the molecules in this rarefied gas of multisolitons form a partition of the eigenvalue parameters of the single approximating multisoliton.  The interrelation of the position parameters $\vec c^j$ is much more subtle since it must accommodate the correct combination of phase-shifts.  

\begin{prop}\label{P:mol}
Let multisoliton parameters $\vec \beta^j$ and $\vec c^{\,j}$ be given for each  $1\leq j\leq J$, with no eigenvalue repeated.  For any $J$-tuple of sequences $x_n^j$ satisfying
\begin{equation}\label{mol apart}
\lim_{n\to \infty} \bigl(x_n^j- x_n^i\bigr)=\infty \quad\text{for all} \quad 1\leq i< j \leq J,
\end{equation}
there exists a sequence $\vec c_n$ so that setting $\vec\beta=\coprod \vec\beta^{j}$, we have
\begin{align}\label{2:22}
Q_{\vec\beta,\vec c_n}(x) - \sum_{j=1}^{J}Q_{\vec\beta^j,\vec c^j}(x-x_n^j) \longrightarrow 0
\end{align}
in $L^2(\R)$ sense as $n\to \infty$.
\end{prop}

The decoupling requirement \eqref{mol apart} could be stated with absolute values without affecting the conclusion of the theorem.  However, ordering the translation  parameters from the start makes the proof much easier to explain.

The scenario analyzed here is something of a reverse of the long-time asymptotics of multisolitons.  In that scenario, one starts with a multisoliton
$Q_{\vec \beta,\vec c(t) }$, with the components of $\vec c(t)$ satisfying an analogue of \eqref{mol apart} as $t\to \infty$, and the goal is to find positions $x^j(t)$ so that $Q_{\vec \beta,\vec c(t) }$ can be approximated by a linear combination of one-solitons as $t\to \infty$.  Despite these differences, we still feel that our approach to treating the error terms could streamline discussions of that subject too. 
 
Each of the multisolitons appearing in \eqref{2:22} is defined via the determinant of a matrix and each matrix is potentially of a different size.  We need a prudent means of indexing all these matrices.  For each $1\leq j\leq N$, let $I^j$ denote (disjoint) index sets of size $\#\vec\beta^j$ (the number of entries in $\vec \beta^j$).  We will then use $I=\coprod I^j$ as our indexing set of size $\#\vec\beta$.

Our first application of these notations is to give a formula for the sequence $c_n$ needed for Proposition~\ref{P:mol}:  For $\mu\in I^j$,
\begin{equation}\label{E:c_n for mol}
 (c_n)_\mu = x_n^j + c^j_\mu -\tfrac{1}{\beta_\mu} \sum_\sigma \log \Bigl[ \tfrac{\beta_{\sigma}-\beta_\mu}{\beta_{\sigma} + \beta_\mu} \Bigr],
\end{equation}
where the sum extends over all $\sigma\in I^\ell$ for all $\ell>j$.

We also need to construct two families of matrices:
For fixed $1\leq j\leq J$ we define a matrix $B^{(j)}(x;\vec\beta,\vec c)$ indexed over $I\times I$ by
\begin{equation}\label{E:B table}
\mbox{\def\s{\vrule depth 1.5ex height 2.7ex width 0mm}
\begin{tabular}{|c|c|c|c|}
\hline
$B_{\mu\nu}^{(j)}(x;\vec\beta,\vec c) \s$ 	& $\nu\in I^\ell,\ \ell<j$ & $\nu\in I^j$ & $\nu\in I^\ell,\ \ell>j$ \\
\hline
$\mu\in I^\ell,\ \ell<j$\s	& $\delta_{\mu\nu}$ & $0$ & $0$ \\
\hline
$\mu\in I^j$\s 				& $0$ & $A_{\mu\nu}(x)$ & $\tfrac{1}{\beta_\mu+\beta_\nu} e^{-\beta_\mu(x-c_\mu)}$ \\
\hline
$\mu\in I^\ell,\ \ell>j$\s	& $0$ & $\tfrac{1}{\beta_\mu+\beta_\nu} e^{-\beta_\nu(x-c_\nu)}$  & $\tfrac{1}{\beta_\mu+\beta_\nu}$ \\
\hline
\end{tabular}
}\end{equation}
where $A_{\mu\nu}(x)$ is as in \eqref{matrix A}.  Similarly, we define
\begin{equation}\label{E:E table}
\mbox{\def\s{\vrule depth 1.5ex height 2.7ex width 0mm}
\begin{tabular}{|c|c|c|c|}\hline
$E_{\mu\nu}^{(j)}(x;\vec\beta,\vec c) \s$ 	& $\nu\in I^\ell,\ \ell<j$ & $\nu\in I^j$ & $\nu\in I^\ell,\ \ell>j$ \\
\hline
$\mu\in I^\ell,\ \ell<j$\s	& $A_{\mu\nu}(x)-\delta_{\mu\nu}$ & $A_{\mu\nu}(x)$ & $\tfrac{1}{\beta_\mu+\beta_\nu} e^{-\beta_\mu(x-c_\mu)}$ \\
\hline
$\mu\in I^j$\s 				& $A_{\mu\nu}(x)$ & $0$ & $0$ \\
\hline
$\mu\in I^\ell,\ \ell>j$\s	& $\tfrac{1}{\beta_\mu+\beta_\nu} e^{-\beta_\nu(x-c_\nu)}$ & $0$  & $\delta_{\mu\nu} e^{\beta_\mu(x-c_\mu)+\beta_\nu(x-c_\nu)}$ \\
\hline
\end{tabular}
}\end{equation}
As we shall see, $B^{(j)}$ is the dominant term for those $x$ near $x_n^k$, while $E^{(j)}$ functions as an error term.

\begin{lemma}\label{L:B&E}
Fix $L>0$.  Then under the hypotheses of Proposition~\ref{P:mol},
\begin{align}\label{B&E to 0}
\limsup_{n\to\infty} \, \bigl\| B^{(j)}(x_n^j+z;\vec\beta,\vec c_n) \bigr\| <\infty \qtq{and}
	\limsup_{n\to\infty}\, \bigl\| E^{(j)}(x_n^j+z;\vec\beta,\vec c_n) \bigr\| =0 
\end{align}
uniformly for $z\in[-L,L]+i[-1,1]$.  Moreover, for all $x\in\R$,
\begin{align}\label{Q from B E}
Q_{\vec \beta,\vec c_n}(x) = - 2\tfrac{d^2}{dx^2} \ln \det\bigl[B^{(j)}(x;\vec\beta,\vec c_n) + E^{(j)}(x;\vec\beta,\vec c_n) \bigr].
\end{align}
\end{lemma}

\begin{proof}
As we are dealing with finite matrices, our claims about the operator norm can be verified considering each matrix entry individually.  From this perspective, the claim \eqref{B&E to 0} follows simply from the behavior of $x_n^j - (c_n)_\mu$: this is bounded when $\mu\in I^j$; it diverges to $+\infty$ when $\mu\in I^\ell$ with $\ell<j$; and it diverges to $-\infty$ when $\mu\in I^\ell$ with $\ell>j$.

The claim \eqref{Q from B E} follows readily from the identity
$$
\det\bigl[ A_{\vec\beta,\vec c_n}(x) \bigr] =
 \det\bigl[B^{(j)}(x;\vec\beta,\vec c_n) + E^{(j)}(x;\vec\beta,\vec c_n) \bigr] \times \prod e^{-2\beta_\mu(x-c_\mu)}, 
$$
where the product is taken over those $\mu\in I^\ell$ for each $\ell<j$.  This product appears because common factors have been extracted from these rows and columns.
\end{proof}

As a stepping-stone to our analysis of $B^{(j)}$ in Lemma~\ref{L:B j}, we first make preparations for evaluating its determinant.  In the case $D_{\mu\nu}\equiv 0$, our next lemma relates two Cauchy determinants (as they are known); indeed, it provides the basic inductive step for the complete evaluation of such determinants.

\begin{lemma}[A Cauchy-like Determinant]\label{L:MCD} Given an $N\times N$ matrix $D$, real numbers $a_1,\ldots,a_N$, and positive $\beta_1,\ldots,\beta_{N+1}$, we define
$$
\tilde a_\mu = \tfrac{\beta_{N+1}-\beta_\mu}{\beta_{N+1} + \beta_\mu} \, a_\mu .
$$
Then we have the following identity between two determinants:
\begin{align}\label{inductive Cauchy}
\begin{vmatrix}
D_{\mu\nu} + \tfrac{a_\mu a_\nu}{\beta_\mu+\beta_\nu} & \tfrac{a_\mu}{\beta_\mu+\beta_{N+1}} \\[2mm] \tfrac{a_\nu}{\beta_{N+1}+\beta_\nu} & \tfrac{1}{\beta_{N+1}+\beta_{N+1}}
\end{vmatrix}
= \tfrac{1}{2\beta_{N+1}} \begin{vmatrix} D_{\mu\nu} + \tfrac{\tilde a_\mu\tilde a_\nu}{\beta_\mu+\beta_\nu} \end{vmatrix} .
\end{align}
On the right, we have an $N\times N$ determinant.  The one on the left is $(N+1)\times (N+1)$, with the extra row and column as indicated.
\end{lemma}

\begin{proof}
This is a simple matter of applying row and column operations:  First we subtract $a_\mu$ times the bottom row of LHS\eqref{inductive Cauchy} from the $\mu^\text{th}$ row and use the identity
$$
\tfrac{1}{\beta_\mu+\beta_\nu} - \tfrac{1}{\beta_{N+1}+\beta_\nu} = \tfrac{\beta_{N+1}-\beta_\mu}{(\beta_{N+1}+\beta_\nu)(\beta_\mu+\beta_\nu)}.
$$
Extracting the common factor from the final column, this yields
$$
\text{LHS\eqref{inductive Cauchy}} = \tfrac{1}{2\beta_{N+1}} \begin{vmatrix} D_{\mu\nu} + \tfrac{\hat a_\mu\check a_\nu}{\beta_\mu+\beta_\nu} & \tfrac{\hat a_\mu}{\beta_\mu+\beta_{N+1}} \\[2mm]
\check a_\nu & 1 \end{vmatrix}
 \text{ with }  
$$
with $\hat a_\mu = (\beta_{N+1}-\beta_\mu) a_\mu$ and $\check a_\nu = \tfrac{a_\nu}{\beta_{N+1} + \beta_\nu}$.

Next we subtract $\check a_\nu$ times the last column from the $\nu^\text{th}$ and apply the identity
$$
\tfrac{1}{\beta_\mu+\beta_\nu} - \tfrac{1}{\beta_\mu + \beta_{N+1}} = \tfrac{\beta_{N+1}-\beta_\nu}{(\beta_\mu+\beta_\nu)(\beta_\mu+\beta_{N+1})} .
$$
The result then follows since the bottom row is now populated by zeros, excepting a one in the final position.
\end{proof}

\begin{lemma}\label{L:B j} Fix $L>0$. Under the hypotheses on Proposition~\ref{P:mol}, there exists $\delta>0$ so that
\begin{align}
\bigl| \, \det\bigl[B^{(j)}(x_n^j+z;\vec\beta,\vec c_n) \bigr] \bigr| &\gtrsim 1  \label{E:B j1}
\end{align}
uniformly for $n\in \N$ and $z\in[-L,L]+i[-\delta,\delta]$.  Moreover, for every $x\in\R$,
\begin{align}
- 2\tfrac{d^2}{dx^2} \ln \det\bigl[B^{(j)}(x;\vec\beta,\vec c_n) \bigr] &= Q_{\vec\beta^j,\vec c^j}(x-x^j_n). \label{E:B j2}
\end{align}
\end{lemma}

\begin{proof}
Applying Lemma~\ref{L:MCD} iteratively, we find that
\begin{align}\label{E:dslfkjs}
\det\bigl[B^{(j)}(x;\vec\beta,\vec c_n) \bigr] = \Bigl(\prod_\sigma \tfrac{1}{2\beta_\sigma}\Bigr)
	\cdot \det\Bigl[ \delta_{\mu\nu} + \tfrac{\tilde a_\mu \tilde a_\nu}{\beta_\mu+\beta_\nu}  \Bigr]_{I^j\times I^j}
\end{align}
where the parameters $\tilde a_\mu$ (which also depend on $n$) are given by
$$
\tilde a_\mu = \exp\{-\beta_\mu(x-(c_n)_\mu)\} \cdot \prod_\sigma \tfrac{\beta_{\sigma}-\beta_\mu}{\beta_{\sigma} + \beta_\mu},
$$
and both products extend over all $\sigma\in I^\ell$ for all $\ell>j$.  Referring back to Definition~\ref{D:multisoliton} and \eqref{E:c_n for mol}, we see that we have succeeded in proving \eqref{E:B j2}.

When $z$ is real, the inequality \eqref{E:B j1} follows from \eqref{E:dslfkjs} because the matrix $C$ with entries
$
C_{\mu\nu}=\tfrac{\tilde a_\mu \tilde a_\nu}{\beta_\mu+\beta_\nu}  
$
is bounded and (strictly) positive definite (as is easily verified from \eqref{inductive Cauchy} and Sylvester's criterion).

To extend the bound to complex $z$, it suffices to show that we can choose $\delta>0$ so that every eigenvalue of the matrix $C$ has positive real part.  Writing $z=x+iy$, we see that it suffices to prove that 
$$
\sum \overline{\psi_\mu}  \bigl[\cos(\beta_\mu y)\cos(\beta_\nu y) - \sin(\beta_\mu y) \sin(\beta_\nu y) \bigr] \tfrac{\tilde a_\mu(x) \tilde a_\nu(x)}{\beta_\mu+\beta_\nu}  \psi_\nu \geq 0
$$
for every complex vector $\psi_\mu$.  Thus, we see that there is such a choice of $\delta>0$ because of the boundedness and positive-definiteness of $C$ for $z$ real. 
\end{proof}

\begin{proof}[Proof of Proposition~\ref{P:mol}]
Our first goal is to prove the following variant of \eqref{2:22}: 
\begin{align}\label{2:22'}
Q_{\vec\beta,\vec c_n}(x+x_n^j) \longrightarrow Q_{\vec\beta^j,\vec c^j}(x) \quad\text{as $n\to\infty$,}
\end{align}
uniformly for $x\in[-L,L]$ for each fixed $j$ and any fixed $L>0$.

Combining Cramer's rule with the Hadamard inequality, we find
\begin{align*}
\bigl\| B^{(j)}(x;\vec\beta,\vec c_n)^{-1} E^{(j)}(x;\vec\beta,\vec c_n) \bigr\|
	&\lesssim  \frac{\bigl\| B^{(j)}(x;\vec\beta,\vec c_n)\bigr\|^{\#\vec\beta-1} \bigr\| E^{(j)}(x;\vec\beta,\vec c_n) \bigr\|}
		{\bigl|\det B^{(j)}(x;\vec\beta,\vec c_n)\bigr|}
\end{align*}
where the implicit constant depends only on $\#\vec \beta$.  Thus, it follows from \eqref{E:B j1} and Lemma~\ref{L:B&E} that for each $L>0$ there is a $\delta>0$ so that
$$
\ln \det\bigl[B^{(j)}(x_n^j+z;\vec\beta,\vec c_n) + E^{(j)}(x_n^j+z;\vec\beta,\vec c_n) \bigr]
	= \ln \det\bigl[B^{(j)}(x_n^j+z;\vec\beta,\vec c_n) \bigr] + o(1)
$$
as $n\to\infty$ uniformly for $z\in [L,-L]+i[-\delta,\delta] $.  Because we have convergence in a \emph{complex} neighbourhood of each $x$, this convergence extends to all derivatives.  Thus \eqref{2:22'} follows from \eqref{Q from B E} and \eqref{E:B j2}.

From \eqref{2:22'} we may then infer that as $n\to\infty$,
\begin{align}\label{on E_n}
\int_{E_n} \Bigl| Q_{\vec\beta,\vec c_n}(x) - \sum_{j=1}^{J}Q_{\vec\beta^j,\vec c^j}(x-x_n^j)\Bigr|^2 \,dx  \longrightarrow 0,
\end{align}
where $E_n = \cup_j [x_n^j-L,x_n^j+L]$.  On the other hand, from \eqref{conserv of multi} we get
$$
\int_\R \bigl| Q_{\vec\beta,\vec c_n}(x)\bigr|^2\,dx = \sum_{\mu\in I} \tfrac{16}{3}\beta_\mu^3 =  \sum_{j=1}^{J} \int \bigl|Q_{\vec\beta^j,\vec c^j}(x-x_n^j)\bigr|^2 \,dx.
$$
Using this and \eqref{on E_n}  we find that the integral over the complementary region $E_n^c$ makes an asymptotically negligible contribution for $L$ large.  Thus \eqref{2:22} follows.
\end{proof}

%%%%%%%%%%%%%%%%%%%%%%%%%%%%%%%%%%%%%%%%%%%%%%%%%%%%%%%%%%%%%%%%%%%%%%%%%%%%%%%%%%%%%%%%%%%%%%%%%%%%%%%%%%%%%%%%%%%%%%%%%%%%%%
\section{Concentration compactness}\label{S:CC}
%%%%%%%%%%%%%%%%%%%%%%%%%%%%%%%%%%%%%%%%%%%%%%%%%%%%%%%%%%%%%%%%%%%%%%%%%%%%%%%%%%%%%%%%%%%%%%%%%%%%%%%%%%%%%%%%%%%%%%%%%%%%%%

The goal of this section is to develop a concentration-compactness principle for the functional $\alpha$ acting on bounded equicontinuous sequences in $H^{-1}$.

\begin{prop}[Concentration compactness principle]\label{P:CC}
Assume that $\{u_n\}_{n \geq 1}$ is a bounded and equicontinuous sequence in $H^{-1}$.  Passing to a subsequence there exist $J^*\in\{0,1,2,\ldots\}\cup \{\infty\}$, non-zero profiles $\phi^j\in H^{-1}$, and positions $x_n^j\in \R$ such that for any finite $0\leq J\leq J^*$ we have the decomposition
$$
u_n(x) = \sum_{j=1}^J \phi^j(x-x_n^j) + r_n^J(x)
$$
with the following properties: for each fixed $\kappa\geq 1+ \sup_n\|u_n\|_{H^{-1}}^2$, 
\begin{align}
&\lim_{J\to J^*}\lim_{n\to \infty} \tr \bigl\{ \bigl(\sqrt{R_0(i\kappa)}r_n^J\sqrt{R_0(i\kappa)}\bigr)^4\bigr\}=0\label{r to 0},\\
&\sup_J\lim_{n\to \infty} \Bigl[\alpha(\kappa; u_n) -  \sum_{j=1}^J\alpha(\kappa;\phi^j) - \alpha(\kappa;r_n^J) \Bigr]= 0,\label{alpha decoupling}\\
&\lim_{n\to \infty} |x_n^j- x_n^\ell|=\infty \quad\text{for all} \quad j\neq \ell. \label{decoup par}
\end{align}
Moreover, 
\begin{align}\label{E:alpha r}
\lim_{J\to J^*}\ \lim_{n\to \infty}\ \biggl| \alpha(\kappa; r_n^J) - \frac1{2\kappa}\int \frac{|\widehat{r_n^J}(\xi)|^2}{\xi^2+4\kappa^2}\, d\xi\biggr|=0.
\end{align}
\end{prop}

Equation \eqref{r to 0} shows that the remainder is small in the sense that a certain operator is small in $\I_4$, the trace ideal modeled on $\ell^4$.  In fact, it is negligible in any $\I_p$ with $p>2$.  This follows from \eqref{r to 0} by means of the basic inequality
\begin{align}\label{Ip Holder}
\| A \|_{\I_p} \leq  \| A \|_{\I_{p_1}}^\theta\| A \|_{\I_{p_2}}^{1-\theta}
	\quad\text{when $1\leq p_1<p<p_2\leq\infty$ and } \theta=\tfrac{p_1}{p}\cdot\tfrac{p_2-p}{p_2-p_1}.
\end{align}

Nonetheless, as \eqref{E:alpha r} shows, the remainder term may make a significant contribution to $\alpha$.  We shall ultimately see that optimizing sequences must have negligible remainder term because it contributes \emph{too much} to $\alpha$.

The nucleus of the proof of Proposition~\ref{P:CC} is the inverse inequality Lemma~\ref{L:inverse}.  It shows that non-trivial $\I_4$ norm may be attributed to the existence of a non-trivial profile common to a subsequence of the original sequence $u_n$.  Before stating this lemma, let us quickly discuss our notations for basic Littlewood-Paley theory; these will be needed in the proof.

For $N \in 2^\Z$, we write $P_N$ for the Fourier multiplier operators defined via a partition of unity adapted to the partition $\{ \xi\in\R : \frac12N < |\xi| \leq 2N\}$ of $\R$. We then define projections onto high and low frequencies via
$$
P_{\leq N}  f  = \sum_{2^\Z\ni M\leq N} P_M f \qtq{and} P_{\geq N}  f  = \sum_{2^\Z\ni M\geq N} P_M f.
$$
One of the key estimates we need is the Bernstein inequality,
$$
\| P_{\leq N} f \|_{L^q} \lesssim N^{\frac1p-\frac1q} \| f \|_{L^p} \qtq{whenever} 1\leq p \leq q \leq\infty.
$$

\begin{lemma}[Inverse inequality]\label{L:inverse}
Assume $\{u_n\}_{n \geq 1}$ are equicontinuous in $H^{-1}$ and satisfy
\begin{align*}
\eps< \liminf_{n\to \infty}\tr\bigl\{ \bigl(\sqrt{R_0(i\kappa)}u_n\sqrt{R_0(i\kappa)}\bigr)^4\bigr\}\quad \text{and} \quad \limsup_{n\to \infty}\|u_n\|_{H^{-1}}<A
\end{align*}
for some positive $\eps$, finite $A$, and some $\kappa\geq 1+A^2$.  Then passing to a subsequence there exist a non-zero profile $\phi\in H^{-1}$ and positions $x_n\in \R$ such that
\begin{align}
& \,\, \, u_n(x+x_n)\rightharpoonup \phi(x) \quad\text{weakly in $H^{-1}$},\notag\\
&\lim_{n\to \infty} \Bigl[\alpha(\kappa; u_n) -\alpha(\kappa;u_n(\cdot+x_n)-\phi)\Bigr] = \alpha(\kappa;\phi) .\label{decoup}
\end{align}
\end{lemma} 

\begin{proof}
Passing to a subsequence, we may assume that for all $n$ we have 
\begin{align*}
\tfrac12\eps<\tr\bigl\{ \bigl(\sqrt{R_0(i\kappa)}u_n\sqrt{R_0(i\kappa)}\bigr)^4\bigr\} \quad \text{and} \quad \|u_n\|_{H^{-1}}^2<2A^2.
\end{align*}

For $N\in 2^\N$, we use \eqref{R I2 kappa} to estimate 
\begin{align}\label{hi}
\tr\bigl\{ \bigl(\sqrt{R_0(i\kappa)}[P_{\geq N}u_n]\sqrt{R_0(i\kappa)}\bigr)^4\bigr\}
&\lesssim  \bigl\| \sqrt{R_0(i\kappa)}[P_{\geq N}u_n]\sqrt{R_0(i\kappa)}\bigr\|_{\I_2}^4\notag\\
&\lesssim \kappa^{-2} \|P_{\geq N}u_n\|_{H^{-1}_\kappa}^4\lesssim \|P_{\geq N}u_n\|_{H^{-1}}^4<\tfrac18\eps,
\end{align}
provided $N$ is sufficiently large depending on $\eps$, in view of the equicontinuity of $u_n$.

On the other hand, for dyadic $N\leq 1$ we may use Bernstein to estimate 
\begin{align}\label{lo}
\tr\bigl\{& \bigl(\sqrt{R_0(i\kappa)}[P_{\leq N}u_n]\sqrt{R_0(i\kappa)}\bigr)^4\bigr\}\notag\\
&\lesssim \bigl\| \sqrt{R_0(i\kappa)}[P_{\leq N}u_n]\sqrt{R_0(i\kappa)}\bigr\|_{\I_2}^2 \bigl\| \sqrt{R_0(i\kappa)}[P_{\leq N}u_n]\sqrt{R_0(i\kappa)} \bigr\|_\op^2\notag\\
&\lesssim \kappa^{-1} \|u_n\|_{H^{-1}_\kappa}^2 \kappa^{-4}\|P_{\leq N} u_n\|_{\infty}^2\lesssim \kappa^{-3} N \|u_n\|_{H^{-1}_\kappa}^4\lesssim N A^4<\tfrac18\eps,
\end{align}
provided $N$ is sufficiently small depending on $\eps$ and $A$.

Therefore, passing to a further subsequence, we deduce that there exists a dyadic $N$ such that
$$
\tr\bigl\{ \bigl(\sqrt{R_0(i\kappa)}[P_Nu_n]\sqrt{R_0(i\kappa)}\bigr)^4\bigr\}\geq c(\eps,A) \qquad \text{for all $n$},
$$
where $c(\eps, A)$ is a positive continuous function on $[0, \infty)\times[0,\infty)$. As
\begin{align*}
\tr\bigl\{ \bigl(\sqrt{R_0(i\kappa)}[P_Nu_n]\sqrt{R_0(i\kappa)}\bigr)^4\bigr\}\lesssim \kappa^{-1} \|u_n\|_{H^{-1}_\kappa}^2 \kappa^{-4} \|P_N u_n\|_{\infty}^2\lesssim A^2 \|P_N u_n\|_{\infty}^2,
\end{align*}
there exists $x_n\in \R$ such that 
\begin{align}\label{nontrivial}
|[P_Nu_n](x_n)|\gtrsim c(\eps,A)A^{-2}.
\end{align}

As the sequence $u_n(x+x_n)$ is bounded in $H^{-1}$, passing to a subsequence we find $\phi\in H^{-1}$ such that
\begin{align}\label{weak convg}
u_n(x+x_n)\rightharpoonup \phi(x) \quad\text{weakly in $H^{-1}$}.
\end{align}
In view of \eqref{nontrivial}, we see that $\phi\neq 0$.  In fact, it is not difficult to verify that
\begin{align}\label{nontrivial rem}
\tr\bigl\{ \bigl(\sqrt{R_0(i\kappa)}\phi\sqrt{R_0(i\kappa)}\bigr)^4\bigr\}\geq \tilde c(\eps,A),
\end{align}
where $\tilde c(\eps, A)$ is a positive continuous function on $[0, \infty)\times[0,\infty)$. Indeed, even the operator norm of $\sqrt{R_0(i\kappa)}\phi\sqrt{R_0(i\kappa)}$ satisfies such a lower bound.

It remains to prove the asymptotic decoupling \eqref{decoup}.  To this end, it suffices to show that for all $\ell\geq 2$ we have 
\begin{align}\label{trace decoup}
\lim_{n\to \infty}\tr\bigl\{ \bigl(\sqrt{R_0(i\kappa)}u_n\sqrt{R_0(i\kappa)}\bigr)^\ell\bigr\} &- \tr\bigl\{ \bigl(\sqrt{R_0(i\kappa)}\bigl[u_n(\cdot+x_n)-\phi\bigr]\sqrt{R_0(i\kappa)}\bigr)^\ell\bigr\}\notag\\
&=\tr\bigl\{ \bigl(\sqrt{R_0(i\kappa)}\phi\sqrt{R_0(i\kappa)}\bigr)^\ell\bigr\}.
\end{align}

The case $\ell=2$ of \eqref{trace decoup} follows easily from the weak convergence \eqref{weak convg} and the fact that $H^{-1}_\kappa$ is a Hilbert space.  Indeed, by \eqref{R I2 kappa},
\begin{align*}
\tr\bigl\{ \bigl(\sqrt{R_0(i\kappa)}&u_n\sqrt{R_0(i\kappa)}\bigr)^2\bigr\} - \tr\bigl\{ \bigl(\sqrt{R_0(i\kappa)}\bigl[u_n(\cdot+x_n)-\phi\bigr]\sqrt{R_0(i\kappa)}\bigr)^2\bigr\}\\
&= \kappa^{-1} \|u_n\|_{H^{-1}_\kappa}^2 - \kappa^{-1} \|u_n(\cdot+x_n)-\phi\|_{H^{-1}_\kappa}^2\\
&= \kappa^{-1}\|\phi\|_{H^{-1}_\kappa}^2 + 8\Re \langle R_0(2i\kappa)\phi, u_n(\cdot+x_n)-\phi\rangle_{L^2}\\
&=\tr\bigl\{ \bigl(\sqrt{R_0(i\kappa)}\phi\sqrt{R_0(i\kappa)}\bigr)^2\bigr\} + o(1) \quad\text{as}\quad n\to \infty.
\end{align*}

We now turn to the case $\ell\geq 3$ in \eqref{trace decoup}.  First, combining \eqref{Ip Holder} with \eqref{hi} and \eqref{lo}, we see that we may discard very high and very low frequencies from further consideration.  
Thus, it suffices to prove \eqref{trace decoup} under the assumption that $u_n$ and $\phi$ are replaced by $P_{\text{med}}u_n$ and $P_{\text{med}}\phi$, respectively.  Passing to a further subsequence, if necessary, in this case we have
\begin{align}\label{to 0}
P_{\text{med}}\bigl[ u_n(\cdot+x_n)-\phi\bigr] \to 0 \quad\text{uniformly on compact sets}.
\end{align}
To continue, we write
\begin{align*}
\tr\bigl\{ \bigl(\sqrt{R_0(i\kappa)}&[P_{\text{med}}u_n]\sqrt{R_0(i\kappa)}\bigr)^\ell\bigr\} - \tr\bigl\{ \bigl(\sqrt{R_0(i\kappa)}[P_{\text{med}}\phi]\sqrt{R_0(i\kappa)}\bigr)^\ell\bigr\}\\
&\qquad \qquad \qquad-\tr\bigl\{ \bigl(\sqrt{R_0(i\kappa)}\bigl[P_{\text{med}}\bigl[u_n(\cdot+x_n)-\phi\bigr]\bigr]\sqrt{R_0(i\kappa)}\bigr)^\ell\bigr\}\\
&=\sum \tr\bigl\{ R_0(i\kappa)F_1R_0(i\kappa)F_2\cdots R_0(i\kappa)F_\ell\bigr\},
\end{align*}
where the sum is over all choices of $F_1, \ldots, F_\ell\in \{P_{\text{med}}\bigl[u_n(\cdot+x_n)-\phi\bigr], P_{\text{med}}\phi\}$ that are not all identical.  We estimate
\begin{align*}
&\bigl|\tr\bigl\{ R_0(i\kappa)F_1R_0(i\kappa)F_2\cdots R_0(i\kappa)F_\ell\bigr\}\bigr|\\
&\lesssim \bigl\|P_{\text{med}}\bigl[u_n(\cdot+x_n)-\phi\bigr]R_0(i\kappa)P_{\text{med}}\phi\bigr\|_{\I_2} \|\sqrt{R_0(i\kappa)}\|_{\op}^2\\
&\quad\times \Bigl[ \bigl\| \sqrt{R_0(i\kappa)}[P_{\text{med}}u_n]\sqrt{R_0(i\kappa)}\bigr\|_{\I_2} + \bigl\| \sqrt{R_0(i\kappa)}[P_{\text{med}}\phi]\sqrt{R_0(i\kappa)}\bigr\|_{\I_2} \Bigr]^{\ell-2}\\
&\lesssim \kappa^{-2-\frac{\ell-2}2}\bigl[ \|u_n\|_{H^{-1}_\kappa} +\|\phi\|_{H^{-1}_\kappa} \bigr]^{\ell-2} \\
&\quad\times\Bigl[\kappa^{-1}\bigl\langle R_0(2i\kappa)(P_{\text{med}}\phi)^2, (P_{\text{med}}[u_n(\cdot+x_n)-\phi])^2\bigr\rangle_{L^2}\Bigr]^{\frac12},
\end{align*}
which converges to zero as $n\to \infty$ in view of \eqref{to 0}.
\end{proof}

We are now ready to complete the

\begin{proof}[Proof of Proposition~\ref{P:CC}]
Fix $\kappa_0= 1+\sup_n\|u_n\|_{H^{-1}}^2$.  We will apply Lemma~\ref{L:inverse} at spectral parameter $\kappa_0$ inductively, extracting one profile at a time.  To start, we set $r_n^0:=u_n$.  Now suppose
we have a decomposition up to level $J\geq 0$ satisfying \eqref{alpha decoupling}. Passing to a subsequence if necessary, we set
\begin{align*}
A_J:=\lim_{n\to\infty} \|r_n^J\|_{\dot H^{-1}} \qtq{and} \eps_J:=\lim_{n\to \infty} \tr\bigl\{ \bigl(\sqrt{R_0(i\kappa_0)}r_n^J\sqrt{R_0(i\kappa_0)}\bigr)^4\bigr\}.
\end{align*}

If $\eps_J=0$, we stop and set $J^*=J$.  If not, we apply Lemma~\ref{L:inverse} at spectral parameter $\kappa_0$ to $r_n^J$.  Passing to a subsequence in $n$, this yields a non-zero profile $\phi^{J+1}\in\dot H^{-1}$ and positions $x_n^{J+1}\in\R$ such that
\begin{align}\label{weak lim}
\phi^{J+1}(x)=\wlim_{n\to\infty}r_n^J\bigl(x+x_n^{J+1}\bigr).
\end{align}

To continue, we define $r_n^{J+1}(x):=r_n^J(x)-\phi^{J+1} \bigl(x-x_n^{J+1}\bigr)$. From Lemma~\ref{L:inverse},
\begin{align*}
\lim_{n\to\infty}\Bigl[\alpha(\kappa_0;r_n^J)-\alpha(\kappa_0;r_n^{J+1}) -\alpha(\kappa_0;\phi^{J+1})\Bigr]=0,
\end{align*}
which combined with the inductive hypothesis gives \eqref{alpha decoupling} at the level $J+1$ and spectral parameter $\kappa_0$.  Moreover, from \eqref{trace decoup} we get
\begin{align*}
\lim_{n\to \infty}\tr\bigl\{ \bigl(\sqrt{R_0(i\kappa_0)}r_n^J\sqrt{R_0(i\kappa_0)}\bigr)^4\bigr\} &- \tr\bigl\{ \bigl(\sqrt{R_0(i\kappa_0)}r_n^{J+1}\sqrt{R_0(i\kappa_0)}\bigr)^4\bigr\}\\
&=\tr\bigl\{ \bigl(\sqrt{R_0(i\kappa_0)}\phi^{J+1}\sqrt{R_0(i\kappa_0)}\bigr)^4\bigr\},
\end{align*}
which combined with \eqref{nontrivial rem} yields
\begin{align}\label{eps rec}
\eps_{J+1}\leq\eps_J- c_{J+1}(\eps_J, A_J)
\end{align}
for some positive function $c_{J+1}$ which is continuous on $[0,\infty)\times[0,\infty)$.
 
If $\eps_{J+1}=0$, we stop and set $J^*=J+1$; in this case, \eqref{r to 0} at spectral parameter $\kappa_0$ is automatic.  If $\eps_{J+1}>0$ we continue the induction.  If the algorithm does not terminate in finitely many steps, we set  $J^*=\infty$; in this case, \eqref{eps rec} guarantees that $\eps_J\to 0$ as $J\to \infty$ and so \eqref{r to 0} at spectral parameter $\kappa_0$ follows.

Next we confirm that \eqref{r to 0} and \eqref{alpha decoupling} hold at \emph{all} spectral parameters $\kappa\geq \kappa_0$.  The asymptotic decoupling \eqref{alpha decoupling} carries over because our argument relies solely on the weak convergence \eqref{weak lim}, as evinced by the proof of \eqref{trace decoup}.  The claim \eqref{r to 0} at spectral parameter $\kappa$ follows from that at spectral parameter $\kappa_0$ since 
\begin{align*}
\bigl\|R_0(i\kappa)^{\frac12}R_0(i\kappa_0)^{-\frac12}\bigr\|_{\op}\leq 1.
\end{align*}

Next we verify the asymptotic orthogonality condition \eqref{decoup par}. We argue by contradiction.  Assume \eqref{decoup par} fails to be true for some pair $(j,\ell)$.  Without loss of generality, we may assume that this is the first pair for which \eqref{decoup par} fails, that is, $j<\ell$ and \eqref{decoup par} holds for all pairs $(j,m)$ with $j<m<\ell$.  Passing to a subsequence, we may assume
\begin{align}\label{cg}
\lim_{n\to \infty}\bigl(x_n^j-x_n^\ell\bigr)= x_0.
\end{align}

From the inductive relation
\begin{align*}
r_n^{\ell-1}=r_n^j-\sum_{m=j+1}^{\ell-1}\phi^m(\cdot-x_n^m),
\end{align*}
we get
\begin{align}\label{tp}
\phi^\ell(x)&=\wlim_{n\to\infty}r_n^{\ell-1}(x+x_n^\ell)\notag\\
&=\wlim_{n\to\infty}r_n^j(x+x_n^\ell)- \sum_{m=j+1}^{\ell-1} \wlim_{n\to \infty}\phi^m(x+x_n^\ell-x_n^m),
\end{align}
where the weak limits are in the $H^{-1}$ topology.  That the first limit on the right-hand side of \eqref{tp} is zero follows from \eqref{cg} and the observation that by construction,
$$
\wlim_{n\to\infty}r_n^j(\cdot+x_n^j)=0.
$$
That the remaining limits are zero follows from our assumption that \eqref{decoup par} holds for all pairs $(j,m)$ with $j<m<\ell$.  Thus \eqref{tp} yields $\phi^\ell=0$, which contradicts the nontriviality of $\phi^\ell$.  This completes the proof of \eqref{decoup par}.

Lastly, we prove \eqref{E:alpha r}:
\begin{align*}
\Bigl| \alpha&(\kappa; r_n^J) - \tfrac1{2\kappa}\int \tfrac{\bigl|\widehat{r_n^J}(\xi)\bigr|^2}{\xi^2+4\kappa^2}\, d\xi\Bigr|\\
&\leq \sum_{\ell\geq 3} \tfrac1\ell \bigl\|\sqrt{R_0(i\kappa)}r_n^J\sqrt{R_0(i\kappa)}\bigr\|_{\I_\ell}^\ell\\
&\leq \bigl\|\sqrt{R_0(i\kappa)}r_n^J\sqrt{R_0(i\kappa)}\bigr\|_{\I_4}^2 \bigl\|\sqrt{R_0(i\kappa)}r_n^J\sqrt{R_0(i\kappa)}\bigr\|_{\I_2} \\
&\qquad\qquad + \sum_{\ell\geq 4} \bigl\|\sqrt{R_0(i\kappa)}r_n^J\sqrt{R_0(i\kappa)}\bigr\|_{\I_4}^4 \bigl\|\sqrt{R_0(i\kappa)}r_n^J\sqrt{R_0(i\kappa)}\bigr\|_{\op}^{\ell-4} \\
&\lesssim \bigl\|\sqrt{R_0(i\kappa)}r_n^J\sqrt{R_0(i\kappa)}\bigr\|_{\I_4}^2 \kappa^{-\frac12}\|r_n^J\|_{H^{-1}_\kappa} \\
&\qquad\qquad +  \bigl\|\sqrt{R_0(i\kappa)}r_n^J\sqrt{R_0(i\kappa)}\bigr\|_{\I_4}^4 \sum_{\ell\geq 4}\kappa^{-\frac{\ell-4}2}\|r_n^J\|_{H^{-1}_\kappa}^{\ell-4}\\
&\lesssim_A \bigl\|\sqrt{R_0(i\kappa)}r_n^J\sqrt{R_0(i\kappa)}\bigr\|_{\I_4}^2 + \bigl\|\sqrt{R_0(i\kappa)}r_n^J\sqrt{R_0(i\kappa)}\bigr\|_{\I_4}^4.
\end{align*}
Thus, using \eqref{r to 0} we deduce \eqref{E:alpha r}
\end{proof}

%%%%%%%%%%%%%%%%%%%%%%%%%%%%%%%%%%%%%%%%%%%%%%%%%%%%%%%%%%%%%%%%%%%%%%%%%%%%%%%%%%%%%%%%%%%%%%%%%%%%%%%%%%%%%%%%%%%%%%%%%%%%%%
\section{Orbital stability}\label{S:OS}
%%%%%%%%%%%%%%%%%%%%%%%%%%%%%%%%%%%%%%%%%%%%%%%%%%%%%%%%%%%%%%%%%%%%%%%%%%%%%%%%%%%%%%%%%%%%%%%%%%%%%%%%%%%%%%%%%%%%%%%%%%%%%%

This section is dedicated to the proof of Theorem~\ref{T:main}.  We argue by contradiction.

Fix $N\geq 1$ and distinct positive parameters $\beta_1, \ldots , \beta_N$.  Assume, towards a contradiction, that there exist $\eps_0>0$, initial data $q_n(0)\in H^{-1}$, and times $t_n\in \R$ such that
\begin{align}\label{data}
\inf_{\vec c\in \R^N} \|q_n(0)- Q_{\vec \beta, \vec c}\|_{H^{-1}}\longrightarrow 0 \quad \text{as}\quad n\to \infty
\end{align}
but 
\begin{align}\label{witness}
\inf_{\vec c\in \R^N} \|q_n(t_n)- Q_{\vec \beta, \vec c}\|_{H^{-1}}\geq \eps_0 \quad \text{for all}\quad n\geq 1.
\end{align}

Recalling that $a_{\ren}$ and $\alpha$ are continuous functions on $H^{-1}$ and conserved by the KdV flow, \eqref{data}, \eqref{a ren Qbc}, and \eqref{alpha Qbc} imply that
\begin{align}
\lim_{n\to \infty}a_{\ren}(k; q_n(t_n))=\lim_{n\to \infty}a_{\ren}(k; q_n(0)) = \prod_{m=1}^N \frac{k-i\beta_m}{k+i\beta_m} e^{\frac{2i\beta_m}{k}}\label{aren convg}
\end{align}
uniformly for $k$ in compact subsets of $\C^+$ and 
\begin{align}
\lim_{n\to \infty}\alpha(\kappa; q_n(t_n))=\lim_{n\to \infty}\alpha(\kappa; q_n(0)) =-\sum_{m=1}^N\ln\bigl(\tfrac{\kappa-\beta_m}{\kappa+\beta_m} \bigr) +\tfrac{2\beta_m}{\kappa}\label{alpha convg}
\end{align}
uniformly for $\kappa\geq 1+ \frac{512}3\sum_m \beta_m^3$.  With a view to future needs, our bound on $\kappa$ combines the restriction needed for \eqref{37deg'} with \eqref{conserv of multi} and the embedding $L^2\hookrightarrow H^{-1}$.  

By Hurwitz's theorem and \eqref{aren convg}, we deduce that for each $1\leq m\leq N$ and $n$ sufficiently large, there exist $\beta_m^{(n)}$ such that
\begin{align}\label{seq of 0}
a_{\ren}\bigl(i\beta_m^{(n)}; q_n(t_n)\bigr)=0 \quad \text{and}\quad \lim_{n\to \infty}\beta_m^{(n)} = \beta_m.
\end{align}

Using \eqref{37deg'}, \eqref{alpha convg}, and the notation from \eqref{G}, we obtain
$$
\|q_n(t_n)\|_{H^{-1}_\kappa}^2\leq 4\kappa \alpha(\kappa; q_n(t_n))\xrightarrow[n\to \infty]{} \sum_{m=1}^N 4\kappa G\bigl( \tfrac{\beta_m}{\kappa}\bigr).
$$
As the right-hand side above converges to zero as $\kappa\to \infty$, we deduce that the sequence $u_n:=q_n(t_n)$ is equicontinuous in $H^{-1}$ and so we may apply Proposition~\ref{P:CC}.  Along a subsequence we may decompose
\begin{align}\label{decomposition}
u_n(x) = \sum_{j=1}^J \phi^j(x-x_n^j) + r_n^J(x)
\end{align}
satisfying the properties \eqref{r to 0} and \eqref{alpha decoupling}.

Our goal is to prove that there are finitely many profiles, each having the shape of a (multi)soliton, and that $r_n^J$ converges to zero in $H^{-1}$.  First, we rule out the possibility of vanishing.  Assume, towards a contradiction, that there are no profiles in \eqref{decomposition} and so $u_n=r_n$. Invoking \eqref{E:alpha r}, we obtain
\begin{align*}
 \tfrac1{2\kappa}\int \tfrac{|\widehat{u_n}(\xi)|^2}{\xi^2+4\kappa^2}\, d\xi \xrightarrow[n\to \infty]{} -\sum_{m=1}^N \ln\Bigl(\tfrac{\kappa-\beta_m}{\kappa+\beta_m}\Bigr) +\tfrac{2\beta_m}{\kappa}=\sum_{m=1}^N G\bigl( \tfrac{\beta_m}{\kappa}\bigr).
\end{align*}
This immediately leads to a contradiction since the functions
$$
\kappa\mapsto \tfrac{\kappa^3}{2\kappa}\int \tfrac{|\widehat{u_n}(\xi)|^2}{\xi^2+4\kappa^2}\, d\xi \quad \text{and}\quad \kappa\mapsto \sum_{m=1}^N \kappa^3 G\bigl( \tfrac{\beta_m}{\kappa}\bigr)
$$
have opposite monotonicity.

Therefore, we may assume that there exists at least one non-trivial profile. From \eqref{HS det2 bnd} and \eqref{R I2}, we have
\begin{align}\label{Barry bdd}
\bigl| 1 - a_{\ren}(k;q)\bigr|\lesssim \exp\bigl\{ C(k) \|q\|_{H^{-1}}^2\bigr\}
\end{align}
with $C(k)$ bounded for $k$ in compact subsets of $\C^+$. Consequently, for $J\geq 1$ fixed, the functions
$$
f_n : k\mapsto a_{\ren}(k;u_n) \exp\Bigl\{ -\tfrac{i}{2k}\int \tfrac{\bigl|\widehat{r_n^J}(\xi)\bigr|^2}{\xi^2-4k^2}\, d\xi \Bigr\}
$$
are holomorphic and locally bounded on $\C^+$.  Invoking Montel's theorem and passing to a subsequence, we find that this sequence converges as $n\to \infty$ to a holomorphic function $f$.  Moreover, by \eqref{seq of 0} we have $f(i\beta_m)=0$ for all $1\leq m\leq N$.

To continue, we combine \eqref{alpha decoupling} with \eqref{E:alpha r} and \eqref{alpha convg} to obtain
\begin{align}\label{upper bdd on alpha phij}
\sum_{j=1}^{J^*} \alpha(\kappa; \phi^j) \leq \sum_{m=1}^NG\bigl( \tfrac{\beta_m}{\kappa}\bigr)
\end{align}
and so by \eqref{alpha to L2},
\begin{align}\label{phi in L2}
\tfrac18 \sum_{j=1}^{J^*}  \|\phi^j\|_2^2=\lim_{\kappa\to \infty}\sum_{j=1}^{J^*} \kappa^3\alpha(\kappa; \phi^j) \leq \sum_{m=1}^N \lim_{\kappa\to \infty} \kappa^3G\bigl( \tfrac{\beta_m}{\kappa}\bigr)<\infty.
\end{align}
Using this and \eqref{Barry bdd}, we see that the function $k\mapsto \prod_{j=1}^{J^*} a_{\ren}(k;\phi^j)$ is well defined and holomorphic on $\C^+$.  Invoking \eqref{alpha decoupling} and \eqref{E:alpha r} one more time, we conclude that 
\begin{align}\label{prod aren phij}
\prod_{j=1}^{J^*} a_{\ren}(k;\phi^j) = f(k) \quad\text{for all} \quad k\in \C^+.
\end{align}

Let $\vec \beta^j$ denote the collection of all zeros of $a_{\ren}(i\kappa;\phi^j)$.  Evidently, $\coprod_j \vec\beta^j$ enumerates the zeros (with multiplicity) of $f(ik)$, which contains each 
$\beta_m$, $1\leq m\leq N$.  Also, by Corollary~\ref{C:min alpha}, 
\begin{align}\label{lower bdd on alpha phij}
\alpha(\kappa, \phi^{j})\geq  \sum_{\beta\in \vec\beta^j} G\bigl( \tfrac{\beta}{\kappa}\bigr).
\end{align}

Contrasting \eqref{upper bdd on alpha phij} and \eqref{lower bdd on alpha phij}, we see that $\coprod_j \vec\beta^j=\vec \beta$ without any repetitions.  Moreover, each $\vec\beta^j$ must be non-empty, for otherwise $\alpha(\kappa, \phi^{j})\equiv 0$ and so $\phi^j\equiv 0$, which is impossible; all profiles are non-zero by construction.

From this we deduce that $J^*$ is finite and, after reviewing \eqref{alpha decoupling} and \eqref{E:alpha r}, that $r_n^{J^*}\to 0$ in $H^{-1}$ sense.  More importantly, the comparison of \eqref{upper bdd on alpha phij} and \eqref{lower bdd on alpha phij} shows that each $\phi^j$ must be an optimizer for the variational problem of Theorem~\ref{T:char} with parameters $\vec \beta^j$.  This theorem then tells us that each $\phi^j$ is indeed a multisoliton.

Putting this all together, we deduce that
\begin{align}\label{1}
u_n(x) = \sum_{j=1}^{J^*}Q_{\vec\beta^j,\vec c^j}(x-x_n^j) +r_n(x) \quad\text{with}\quad \lim_{n\to \infty}\| r_n\|_{H^{-1}}=0.
\end{align}

In view of Proposition~\ref{P:mol}, this contradicts \eqref{witness} and so completes the proof of Theorem~\ref{T:main}.\qed

%%%%%%%%%%%%%%%%%%%%%%%%%%%%%%%%%%%%%%%%%%%%%%%%%%%%%%%%%%%%%%%%%%%%%%%%%%%%%%%%%%%%%%%%%%%%%%%%%%%%%%%%%%%%%%%%%%%%%%%%%%%%%%
\section{Higher regularity}\label{S:Hs}
%%%%%%%%%%%%%%%%%%%%%%%%%%%%%%%%%%%%%%%%%%%%%%%%%%%%%%%%%%%%%%%%%%%%%%%%%%%%%%%%%%%%%%%%%%%%%%%%%%%%%%%%%%%%%%%%%%%%%%%%%%%%%%

The purpose of this section is to demonstrate two methods by which one may deduce orbital stability at higher regularity from Theorem~\ref{T:main}.  The two methods are completely independent and so we divide the proof of Corollary~\ref{C:Hs} into two parts:

\begin{proof}[Proof of Corollary~\ref{C:Hs} when $s\in\{0,1\}$]  We begin with the case $s=0$.  Using the conservation of momentum, we find that for any pair of solutions $q(t)$ and $Q(t)$,
\begin{align*}
\| q(t) - Q(t) \|_{L^2}^2 &=  \| q(t) \|_{L^2}^2 - \| Q(t) \|_{L^2}^2 - 2 \langle q(t) - Q(t),\,Q(t)\rangle \\
&\leq \| q(0) \|_{L^2}^2 - \| Q(0) \|_{L^2}^2 + 2 \|q(t) - Q(t)\|_{H^{-1}} \| Q(t) \|_{H^1}.
\end{align*}
Thus, recalling that the momentum of a multisoliton is determined by $\vec\beta$ alone, we see that
\begin{align*}
\inf_{\vec c} \| q(t) - Q_{\vec \beta,\vec c} \|_{L^2}^2  &\leq \inf_{\vec c}  \| q(0) - Q_{\vec \beta,\vec c}\|_{L^2}\bigl( \|q(0)\|_{L^2}+ \sup_{\vec c}\| Q_{\vec \beta,\vec c} \|_{L^2}\bigr)\\
&\quad + 2 \inf_{\vec c} \|q(t) - Q_{\vec \beta,\vec c}\|_{H^{-1}} \cdot \sup_{\vec c} \| Q_{\vec \beta,\vec c} \|_{H^1} .
\end{align*}
As observed by Lax \cite{MR0369963}, the first two polynomial conservation laws control the $H^1$ norm. Thus, the $s=0$ case of Corollary~\ref{C:Hs} follows from \eqref{conserv of multi} and Theorem~\ref{T:main}.

Turning now to the case of $H^1$ we need one preliminary: by Sobolev embedding,
\begin{align*}
\| f\|_{L^3} \lesssim  \||\nabla|^{\frac16}f\|_{L^2}\lesssim \|f \|_{H^1}^{\frac7{12}} \|f\|_{H^{-1}}^{\frac5{12}}.
\end{align*}
Proceeding as we did in the $L^2$ case, but using conservation of energy,
\begin{align*}
\| q'(t) - Q'(t) \|_{L^2}^2 &=  2 H\bigl(q(t)\bigr) - 2 H\bigl(Q(t)\bigr) - 2 \langle q'(t) - Q'(t),\,Q'(t)\rangle \\
	& \qquad -  \int q(t,x)^3 - Q(t,x)^3\,dx \\
&\lesssim 2 \bigl| H\bigl(q(0)\bigr) - 2 H\bigl(Q(0)\bigr) \bigr| + \| q(t) - Q(t) \|_{H^{-1}} \| Q(t) \|_{H^3} \\
	& \qquad + \bigl( \| q(t) \|_{H^1} + \| Q(t) \|_{H^1} \bigr)^{2+\frac7{12}} \|q(t) - Q(t)\|_{H^{-1}}^{\frac5{12}}.
\end{align*}
Thus the result now follows as before from Theorem~\ref{T:main} and the bounds in \cite{MR0369963}.
\end{proof}

The key observation for our second method is the equicontinuity of orbits under \eqref{KdV}.  The specific formulation we need is as follows.

\begin{lemma}\label{L:equiL} Fix $s\in[-1,1)$ and distinct positive parameters $\beta_1, \ldots, \beta_N$.   For every $\eps>0$, there exist $\delta>0$ and $N\in 2^{\Z}$ so that
\begin{align}\label{E:equiL}
\inf_{\vec c} \| q(0) - Q_{\vec\beta,\vec c} \|_{H^s} < \delta \ \implies\ \sup_{t\in\R}  \| q_{\geq N}(t) \|_{H^s} < \eps.
\end{align}
\end{lemma}

\begin{proof} If this assertion were to fail, then there would exist a sequence of solutions $q_n$ and a sequence of times $t_n$ so that
$$
\limsup_{n\to\infty} \ \inf_{\vec c} \| q_n(0) - Q_{\vec\beta,\vec c} \|_{H^s} =0,
$$
but $\{q_n(t_n):n\in \mathbb{N}\}$ is not equicontinuous in $H^s$.

As $\vec c$ varies, the multisolitons $Q_{\vec \beta,\vec c}$ remain uniformly bounded in $H^1$.  Thus this family is $H^s$-equicontinuous and then so must be the sequence of initial data $q_n(0)$.

When $s=-1$, this directly contradicts the equicontinuity result \cite[Prop.~4.4]{KV}.  The analogous equicontinuity result for $-1<s<0$ appears in the proof of \cite[Cor.~5.3]{KV}.  Finally, when $0\leq s <1$, we may appeal to \cite[Prop.~3.6]{KVZ}.  While this last-quoted result does not explicitly assert equicontinuity, the simplicity with which it may be derived from what is presented there is illustrated (in the $s=0$ case) in \cite[Prop.~A.3(c)]{KV}.
\end{proof}

It remains to present the

\begin{proof}[Proof of Corollary~\ref{C:Hs} when $s\in(-1,1)$]  For any $N\in 2^\Z$ and any pair of solutions,
\begin{align*}
\| q(t) - Q(t) \|_{H^s}^2 &\lesssim   N^{2+2s} \| q(t) - Q(t) \|_{H^{-1}}^2 + N^{2s-2}\| P_{\geq N} Q(t) \|_{H^1}^2 + \| P_{\geq N} q(t) \|_{H^s}^2.
\end{align*}
The result now follows from Theorem~\ref{T:main} and Lemma~\ref{L:equiL}.
\end{proof}

\end{document}